\documentclass[reqno,11pt]{amsart}
\usepackage{hyperref}
\usepackage{a4wide,amsthm,amssymb,mathrsfs,setspace,pstricks,booktabs,mathtools,amsmath,amsfonts}
\usepackage{graphicx}
\usepackage{xcolor}
\usepackage{array}
\newtheorem{theorem}{Theorem}[section]
\newtheorem{lemma}[theorem]{Lemma}
\newtheorem{proposition}[theorem]{Proposition}
\newtheorem{corollary}[theorem]{Corollary}

\theoremstyle{definition}

\theoremstyle{remark}
\newtheorem{remark}[theorem]{Remark}

\numberwithin{equation}{section}

\newcommand{\e}{\varepsilon}
\newcommand{\R}{\mathbb{R}}

\newcommand{\Rn}{\mathbb{R}^n}

\newcommand{\Hn}{\mathcal{H}^{n-1}}

\newcommand{\ue}{u_\varepsilon}
\newcommand{\Oe}{\Omega_\varepsilon}

\providecommand{\ISBN}[1]{\textsc{isbn}: #1}  

\begin{document}
	
\title [Sharp stability on the second Robin eigenvalue]{Sharp stability on the second Robin eigenvalue with negative boundary parameters}

	\author{Zhijie Chen}
	\address{Department of Mathematical Sciences, Yau Mathematical Sciences Center, Tsinghua University, Beijing, 100084, China}
	\email{zjchen2016@tsinghua.edu.cn}

	\author{Zhen Song}    
	\address{Department of Mathematical Sciences, Tsinghua University, Beijing 100084, China}
	\email{sz21@mails.tsinghua.edu.cn}

	\author{Wenming Zou}
	\address{Department of Mathematical Sciences, Tsinghua University, Beijing 100084, China}
	\email{zou-wm@mail.tsinghua.edu.cn}



	\begin{abstract}
		In this paper, we prove a quantitative refinement of the isoperimetric type inequality for the second Robin eigenvalue with negative boundary parameters established by Freitas and Laugesen [\emph{Amer.~J.~Math.} \textbf{143} (2021), no.~3, 969--994]. 
		By constructing a suitable family of nearly spherical domains, we prove that the exponent for the Fraenkel asymmetry in this quantitative type inequality is sharp.
	\end{abstract}
	
 	\maketitle

	\section{Introduction and main results}
	\renewcommand{\thethm}{\Alph{thm}}
	
	Let $\Omega$ be a bounded Lipschitz domain in $\R^n$, $n\geq2$. Consider the Robin eigenvalue problem for the Laplace operator on $\Omega$
	\begin{equation}\label{robin problem}
		\begin{cases}
			-\Delta u=\lambda u \quad &\hbox{in}\ \Omega, \\
			\frac{\partial u}{\partial \nu}+\alpha u=0 \quad &\hbox{on}\ \partial\Omega, \\
		\end{cases}
	\end{equation}
where $\nu:\partial\Omega\to\R^n$ is the outer unit normal and $\alpha\in\R$ is the boundary parameter. It is well-known that the eigenvalues $\{\lambda_k(\Omega;\alpha)\}_{k=1}^\infty$ are increasing and continuous as functions of the boundary parameter $\alpha$, and satisfy
$$
\lambda_1(\Omega;\alpha)<\lambda_2(\Omega;\alpha)\leq\lambda_3(\Omega;\alpha)\leq\cdots,\quad \lambda_k(\Omega;\alpha)\to+\infty\quad\text{as}\quad k\to+\infty,
$$
for each fixed $\alpha$. The Robin spectrum connects the Neumann ($\alpha=0$), Dirichlet ($\alpha\to\infty$) and Steklov ($\lambda=0$) eigenvalue problems, and has received great interest in the past years.\vspace{0.5em}

An interesting and important topic for the Robin eigenvalue problem is the shape optimization under a fixed volume constraint. When the boundary parameter $\alpha$ is positive, a Faber-Krahn type inequality is valid for the first Robin eigenvalue, which was proved by Bossel \cite{Bossel} in the two-dimensional case and by Daners \cite{Daners=MathAnn} in the $n$-dimensional case. To be more precise, they proved that if $B$ is a ball of the same volume as $\Omega$, then
\begin{equation}\label{classic faberkrahn}
	\lambda_1(\Omega;\alpha)\geq \lambda_1(B;\alpha).
\end{equation}
For alternative approaches, we refer the readers to \cite{Bucur=CVPDE,Bucur=ARMA2010,Bucur=ARMA2015}. When the boundary parameter $\alpha$ is negative, things become more involved. In 2015, it was proved by Ferone, Nitsch and Trombetti \cite{Ferone=CPAA} that the ball is a local maximizer for the first Robin eigenvalue. In the same year Freitas and Krejčiřík \cite{Freitas=Adv} proved that in the two-dimensional case, the disk is a global maximizer among planar domains for all $\alpha\in [\alpha_*,0)$, where $\alpha_*$ is a negative number depending on the area. Equivalently, they actually derived a reverse Faber-Krahn type inequality. However, in the $n$-dimensional case, they \cite{Freitas=Adv} also proved that the ball cannot remain a global maximizer for large negative boundary parameter $\alpha$ by showing that the first Robin eigenvalue of a spherical shell is bigger than the first Robin eigenvalue of the ball under the same volume constraint.\vspace{0.5em}

For the positive boundary parameter $\alpha>0$, we know the equality in the Bossel-Daners inequality \eqref{classic faberkrahn} holds if and only if $\Omega$ is a ball, which naturally leads to the question of stability. The first quantitative estimate was given by Bucur, Ferone, Nitsch and Trombetti \cite{Bucur=JDE}. Recall that the Fraenkel asymmetry of $\Omega$ is defined by (see for instance \cite{Brasco=GAFA,Maggi=Ann Math})
\begin{equation}\label{Fraenkel}
	\mathcal{A}(\Omega)\coloneqq\min\left\{\frac{|\Omega\Delta B|}{|B|}:\,B\text{ is a ball},\,|B|=|\Omega|\right\},
\end{equation}
where $\Omega\Delta B:=(\Omega\setminus B )\sqcup (B\setminus\Omega)$. Then
they \cite{Bucur=JDE} proved that there exists a positive constant $c=c(n,\alpha,|\Omega|)$ depending only on $n$, $\alpha$ and $|\Omega|$, such that
\begin{equation}\label{classic faberkrahn quantitative}
	\lambda_1(\Omega;\alpha)- \lambda_1(B;\alpha)\geq c(n,\alpha,|\Omega|)\mathcal{A}(\Omega)^2,
\end{equation}
where $B$ is a ball of the same volume as $\Omega$. They also proved that the power $2$ for the Fraenkel asymmetry is sharp by using the perturbation of the ball in ellipsoids. We also refer the interested readers to \cite{Amato talenti} for alternative approach to establish \eqref{classic faberkrahn quantitative} in dimension 2, using a quantitative Talenti-type comparison result with Robin boundary conditions. For the negative boundary parameter $\alpha<0$, whether such quantitative inequality exists is still open. We refer the readers to \cite{Cito=ESAIM} for a quantitative form of the reverse Faber-Krahn type inequality for the first Robin eigenvalue among convex sets of prescribed perimeter, not volume.\vspace{0.5em}

Now let us turn to the second Robin eigenvalue.
For the positive boundary parameter $\alpha>0$, the minimizer is, as in the Dirichlet case, the disjoint union of two equal balls \cite{Kennedy=PAMS,Kennedy=ZAMP}. More concretely, let $\Omega\subset\R^n$ be a bounded Lipschitz domain  and $B_1$, $B_2$ be two disjoint equal balls with volume $|\Omega|/2$. Then for any $\alpha>0$,
\begin{equation}\label{Faberkrahn 2}
	\lambda_2(\Omega;\alpha)\geq\lambda_2(B_1\sqcup B_2;\alpha)
\end{equation}
with equality if and only if $\Omega$ is itself a disjoint union of two equal balls.\vspace{0.5em}

For $\alpha<0$, it was conjectured by Bucur, Freitas and Kennedy \cite[Open Problem 4.41]{Henrot=2017Book} that $\lambda_2(\Omega;\alpha)$ should be maximal for the ball on a range of values $(\alpha^*,0)$ for some $\alpha^*<0$.  In 2021, a significant progress on this open problem has been made by Freitas and Laugesen \cite{Laugesen=AJM} by using Weinberger's classic mass transplantation argument. 

 \vspace{0.5em}
\noindent{\bf Theorem A.} \cite{Laugesen=AJM}
{\it Let $\Omega\subset\R^n$ be a bounded Lipschitz domain and $B$ be a ball of the same volume as $\Omega$, then
\begin{equation}\label{Laugesen}
	\lambda_2(\Omega;\alpha)\leq\lambda_2(B;\alpha),\quad\text{for }\alpha\in\left[-\frac{n+1}{n}R^{-1},0\right],
\end{equation}
where $R$ is the radius of $B$. Moreover, the equality holds if and only if $\Omega$ is a ball.}
 \vspace{0.5em}

Motivated by \eqref{classic faberkrahn quantitative}, it would be of great significance to refine the Freitas-Laugesen inequality \eqref{Laugesen} with a remainder term. To our best knowledge, this stability type question for the second Robin eigenvalue has not been settled in the literature and remains open.\vspace{0.5em}

In this paper, we will refine the Freitas-Laugesen inequality \eqref{Laugesen} with a quantitative term involving the Fraenkel asymmetry. The main result of this paper is the following quantitative Freitas-Laugesen inequality.

\begin{theorem}\label{main result one}
	Let $\Omega$ be a bounded Lipschitz domain in $\R^n$, $n\geq2$, $B$ be a ball of the same volume as $\Omega$, and denote $R$ the radius of $B$. Then there exist a positive constant $\gamma=\gamma(n,\alpha,R)$ depending only on $n$, $\alpha$ and $R$, such that
	\begin{equation}\label{main result one inequality}
		\lambda_2(B;\alpha)-\lambda_2(\Omega;\alpha)\geq\gamma(n,\alpha,R)\mathcal{A}(\Omega)^2,\quad\text{for }\alpha\in \left(-R^{-1}, 0\right).
	\end{equation}
\end{theorem}
Notice that $\lambda_2(\Omega;0)=\mu_2(\Omega)$ is the first non-trivial Neumann eigenvalue. Letting $\alpha\to 0^{-}$ in \eqref{main result one inequality},  we can recover the quantitative Szeg\"o-Weinberger inequality that was established first by Brasco and Pratelli \cite{Brasco=GAFA}.
\begin{corollary}\label{main result two}
	Let $\Omega$ be a bounded Lipschitz domain in $\R^n$, $n\geq2$, and $B$ be a ball of the same volume as $\Omega$. Let $\mu_2(\Omega)>0$ be the first non-trivial Neumann eigenvalue of the Laplace operator on $\Omega$, then
	\begin{equation}\label{main result two inequality}
		\mu_2(B)-\mu_2(\Omega)\geq\delta(n)|\Omega|^{-2/n}\mathcal{A}(\Omega)^2,
	\end{equation}
where $\delta=\delta(n)$ is a positive constant depending only on $n$.
\end{corollary}
\begin{remark} The constants $\gamma(n,\alpha,R)$ in Theorem \ref{main result one} and $\delta(n)$ in Corollary \ref{main result two} are known explicitly, see Section 3.
	Since the scale invariance is not preserved in the Robin eigenvalue problem due to the presence of the boundary parameter $\alpha$, the constants in \eqref{classic faberkrahn quantitative} and \eqref{main result one inequality} depend on the volume of $\Omega$ or equivalently on $R$ (notice that $|\Omega|=|B|=\omega_n R^n$ where $\omega_n$ is the volume of the unit ball in $\R^n$).
\end{remark}
\begin{remark}
	Compared Theorem \ref{main result one} with Theorem A, the reason why we can not prove \eqref{main result one inequality} for $\alpha\in [-\frac{n+1}{n}R^{-1}, -R^{-1}]$ is
	$$\lambda_2(B;\alpha)\begin{cases}>0\quad\text{if }\; \alpha>-R^{-1},\\
	=0\quad\text{if }\; \alpha=-R^{-1},\\
	<0\quad\text{if }\; \alpha<-R^{-1},\end{cases}$$
	which was proved in \cite{Laugesen=AJM} (See Proposition \ref{second} below). One will see in Section 3 that $\lambda_2(B;\alpha)>0$ is important in our proof of \eqref{main result one inequality}. 
\end{remark}
Whether the exponent $2$ for the Fraenkel asymmetry $\mathcal{A}(\Omega)$ in the quantitative Freitas-Laugesen inequality \eqref{main result one inequality} is sharp or not is of great importance. For our problem, it is much more difficult than the one for the first Robin eigenvalue problem with positive boundary parameters because $\lambda_2(B;\alpha)=\cdots=\lambda_{n+1}(B;\alpha)$ is of multiplicity $n$ and thus may be not differentiable. Similar phenomena have been observed in \cite{Brasco=Jfa} and \cite{Brasco=GAFA} before. For the Robin problem, some extra difficulties appear since the first eigenfunction is not constant, which makes this problem more tricky. We must choose the test domains with better properties to prove the sharpness of the exponent $2$, and the last main result of this paper is as follows.
\begin{theorem}\label{main result three}
	The power $2$ for the Fraenkel asymmetry $\mathcal{A}(\Omega)$ in the quantitative Freitas-Laugesen inequality \eqref{main result one inequality} is sharp.
\end{theorem}

Theorem \ref{main result one} and Theorem \ref{main result three} can be viewed as a sharp quantitative refinement of the Freitas-Laugesen inequality \eqref{Laugesen}, and also as a generalization of the stability result \eqref{main result two inequality} on the second Neumann eigenvalue in \cite{Brasco=GAFA}. 
We also refer the interested readers to \cite{Brasco=Jfa,Brasco=Duke,Brasco=GAFA,Maggi=Pisa,FuscoZhang-CVPDE} for the quantitative estimates of other types of eigenvalues for the Laplace operator or $p$-Laplace operator.\vspace{0.5em}

The paper is organized as follows. In Section 2, we give some preliminary results that we will use throughout the paper. Then in Section 3, we will show how to get a quantitative form of the Freitas-Laugesen inequality \eqref{Laugesen} and hence establish Theorem \ref{main result one} and Corollary \ref{main result two}. Finally, the last section is devoted to the construction of a family of nearly spherical domains and then to prove the sharpness of the exponent 2 for the Fraenkel asymmetry in the quantitative  Freitas-Laugesen inequality \eqref{main result one inequality}.

	
\section{Preliminaries on the Robin eigenvalue problem on ball}
We recall some basic facts about the Robin eigenvalues of the Laplace operator on ball. Let $\mathbb{B}\subset\R^n$ be the unit ball from \cite{Laugesen=AJM}. The following two propositions give the first two Robin eigenvalues of $\mathbb{B}$ for every real boundary parameter $\alpha$.
\begin{proposition}[First Robin eigenfunction of the ball] \cite{Laugesen=AJM}\label{first}
	The first Robin eigenvalue of the unit ball $\mathbb{B}\subset\Rn$ is simple and the first eigenfunction is radial for each $\alpha\in\R$.
	\begin{itemize}
		\item If $\alpha<0$, then $\lambda_1(\mathbb{B} ; \alpha)<0$ and the eigenfunction $g(r)$ is positive and radially strictly increasing, with
		$$
		g(0)>0, \quad g^{\prime}(0)=0, \quad g^{\prime}(r)>0, \quad r \in(0,1) .
		$$
		\item  If $\alpha=0$, then $\lambda_1(\mathbb{B} ; 0)=0$, with constant eigenfunction $g(r) \equiv 1$.\vspace{0.2em}
		\item  If $\alpha>0$, then $\lambda_1(\mathbb{B} ; \alpha)>0$ and the eigenfunction $g(r)$ is positive and radially strictly decreasing, with
		$$
		g(0)>0, \quad g^{\prime}(0)=0, \quad g^{\prime}(r)<0, \quad r \in(0,1) .
		$$
	\end{itemize}
\end{proposition}

\begin{proposition}[Second Robin eigenfunctions of the ball]\cite{Laugesen=AJM}\label{second}
The second Robin eigenfunctions of the unit ball $\mathbb{B}\subset\Rn$ have the form $g(r) x_i / r$ for $i=1, \ldots, n$. The radial part $g$ has $g(0)=$ $0, g^{\prime}(0)>0, g(r)>0$ for $r \in(0,1)$, and when $\alpha \leq 0$ it is strictly increasing, with $g^{\prime}(r)>0$.
\begin{itemize}
	\item If $\alpha<-1$, then	
	$$
	\lambda_2(\mathbb{B} ; \alpha)=\cdots=\lambda_{n+1}(\mathbb{B} ; \alpha)<0,
	$$	
	and $r g^{\prime}(r)+\alpha g(r)<0$ for $r \in(0,1)$.\vspace{0.2em}
	\item If $\alpha=-1$, then $g(r)=r$ and	
	$$
	\lambda_2(\mathbb{B} ;-1)=\cdots=\lambda_{n+1}(\mathbb{B} ;-1)=0 .
	$$\vspace{0.2em}
	\item If $\alpha>-1$, then	
	$$
	\lambda_2(\mathbb{B} ; \alpha)=\cdots=\lambda_{n+1}(\mathbb{B} ; \alpha)>0,
	$$	
	and $r g^{\prime}(r)+\alpha g(r)>0$ for $r \in(0,1)$.
\end{itemize}
\end{proposition}
Let $B=B_R(0)\subset\Rn$ be a ball of radius $R$. We restrict ourselves to the case $-R^{-1}<\alpha<0$ throughout this paper. A direct computation shows that the scaling relation for the Robin eigenvalue problem is 
$$
\lambda_k(\Omega;\alpha)=t^{-2}\lambda_k(t^{-1}\Omega;t\alpha),\quad t>0,\quad k\in\mathbb{N}_+.
$$
From here and the above two propositions, we have $$\lambda_1(B ; \alpha)< 0,\quad\lambda_2(B ; \alpha)> 0,\quad\text{for }\;-R^{-1}<\alpha<0.$$ Besides, the second Robin eigenfunctions of $B$ have the form $g(r)x_i/r$ for $i=1,\cdots,n$, and the expression of the radial part $g$ is
\begin{equation}\label{g}
	g(r)=r^{1-n/2}J_{n/2}\left(\sqrt{\lambda_2(\mathbb{B};R\alpha)}\frac{r}{R}\right),\quad\text{for }\;-R^{-1}<\alpha<0,
\end{equation}
where $J_{n/2}$ is the Bessel function of the first kind of order $n/2$. See \cite[page 983]{Laugesen=AJM} for the proof.

\section{Stability for the second Robin eigenvalue}
In this section, we recall the proof of the Freitas-Laugesen inequality \eqref{Laugesen} for $-R^{-1}<\alpha<0$ from \cite{Laugesen=AJM} and then establish the quantitative Freitas-Laugesen inequality \eqref{main result one inequality}. Let 
  $$
  R\coloneqq\left(\frac{|\Omega|}{\omega_n}\right)^{1/n},
  $$
where $\omega_n$ is the volume of unit ball in $\R^n$, so that $|\Omega|=|B_R(0)|$. For convenience, denote $B=B_R(0)$.
Recall the Rayleigh quotient
\begin{align}\label{Rayleigh}Q[u]=Q[u;\Omega,\alpha]=\frac{\int_{\Omega}|\nabla u|^2\,dx+\alpha\int_{\partial\Omega}u^2\,d\mathcal{H}^{n-1}}{\int_{\Omega} u^2\,dx},\quad u\in H^1(\Omega).\end{align}
Then it is well known that the following variational principle holds
\begin{align}\label{variation2}
\lambda_2(\Omega; \alpha)&=\min_{E\subset H^1(\Omega) \atop \text{subspace of dim}\, 2}\max_{u\in E\setminus\{0\}}Q[u;\Omega,\alpha]
=\min_{u\in H^1(\Omega)\setminus\{0\}\atop \int_{\Omega}uu_1\,dx=0}Q[u;\Omega,\alpha],
\end{align}
where $u_1\geq0$ is the eigenfunction for the first Robin eigenvalue $\lambda_1(\Omega;\alpha)$.
\vspace{0.5em}

	If $\lambda_2(\Omega;\alpha)\leq 0$, then there is nothing to prove since $\lambda_2(B;\alpha)>0$. Thus we may assume that $\lambda_2(\Omega;\alpha)>0$. Following \cite{Laugesen=AJM}, we define the test functions by
	\begin{equation}\label{test function}
		\varphi_{i}(x)\coloneqq g(r)\frac{x_i}{r},\quad i=1,\cdots,n,
	\end{equation}
where $g(r)$ is the radial part of the second Robin eigenfunction of the ball $B$, given by \eqref{g}. In particular, $g'(R)=-\alpha g(R)$. Now, we extend it outside the ball $B$ by
\begin{equation}\label{g extend}
	g(r)\coloneqq g(R)e^{-\alpha(r-R)},\quad\text{for}\; r\geq R.
\end{equation}
Then $g\in C^1$.
By the center of mass result \cite[Proposition 2]{Laugesen=AJM}, up to a translation of $\Omega$ we can assume that
\begin{equation}\label{orthogonal condition}
	\int_\Omega\varphi_i u_1\,dx=0,\quad\forall i=1,\cdots,n,
\end{equation}
where $u_1\geq0$ is the eigenfunction for the first Robin eigenvalue $\lambda_1(\Omega;\alpha)$.
Then by the variational principle \eqref{variation2}, we have
\begin{equation}
	\lambda_2(\Omega;\alpha)\int_\Omega\varphi_i^2\,dx\leq\int_\Omega|\nabla \varphi_i|^2\,dx+\alpha\int_{\partial \Omega}\varphi_i^2\,d\mathcal{H}^{n-1},\quad i=1,\cdots,n,
\end{equation}
with equality when $\Omega=B$ (Because Proposition \ref{second} says that $\varphi_i$'s are precisely the eigenfunctions of $\lambda_2(B;\alpha)$). Summing over $i$ and substituting the definition of $\varphi_i$ for $i=1,\cdots,n$, we have
\begin{equation}
	\lambda_2(\Omega;\alpha)\int_\Omega g(r)^2\,dx\leq\int_\Omega\left(g^\prime(r)^2+(n-1)r^{-2}g(r)^2\right) \,dx+\alpha\int_{\partial \Omega}g(r)^2\,d\mathcal{H}^{n-1}.
\end{equation}
By \cite[Proposition 1]{Laugesen=AJM} which says that $$\int_{\partial \Omega}g(r)^2\,d\mathcal{H}^{n-1}\geq \int_{\Omega}\left(2g(r)g'(r)+\frac{n-1}{r}g(r)^2\right)\,dx$$ and the fact that $\alpha<0$, we obtain
\begin{equation}\label{starting point}
	\lambda_2(\Omega;\alpha)\int_\Omega g(r)^2\,dx\leq\int_\Omega h(r)\,dx
\end{equation}
with equality when $\Omega=B$, where
\begin{equation}\label{h}
	h(r)\coloneqq g^\prime(r)^2+(n-1)r^{-2}g(r)^2+2\alpha g(r)g^\prime(r)+\alpha\frac{n-1}{r}g(r)^2.
\end{equation}
Since $-R^{-1}<\alpha<0$, it was proved in \cite[Lemma 8]{Laugesen=AJM} that $g(r)^2$ is strictly increasing and $h(r)$ is strictly decreasing for $0<r<\infty$, which can be done by Proposition \ref{second} and a direct derivative calculation. Then, since $|\Omega\setminus B|=|B\setminus \Omega|$, as proved in \cite[Proposition 3]{Laugesen=AJM}, we have
\begin{align}\label{gr2}
\int_{\Omega}g(r)^2dx&=\int_{\Omega\cap B}g(r)^2dx+\int_{\Omega\setminus B}g(r)^2dx
\nonumber\\
&\geq \int_{\Omega\cap B}g(r)^2dx+|\Omega\setminus B|g(R)^2\nonumber\\
&\geq \int_{\Omega\cap B}g(r)^2dx+\int_{B\setminus \Omega}g(r)^2dx=\int_{B}g(r)^2dx,
\end{align}
and similarly $$\int_{\Omega}h(r)\,dx\leq \int_{B}h(r)\,dx.$$
Therefore, combining these with \eqref{starting point} and $\lambda_2(\Omega;\alpha)>0$, there holds
\begin{equation}\label{use}
	\begin{aligned}
		\lambda_2(\Omega;\alpha)\int_{B}g(r)^2\,dx&\leq\lambda_2(\Omega;\alpha)\int_\Omega g(r)^2\,dx \leq\int_\Omega h(r)\,dx\\
		&\leq\int_{B} h(r)\,dx=\lambda_2(B;\alpha)\int_{B}g(r)^2\,dx,
	\end{aligned}
\end{equation}
from which we obtain the following isoperimetric inequality for the second Robin eigenvalue
\begin{equation}\label{target}
	\lambda_2(\Omega;\alpha)\leq \lambda_2(B;\alpha),\quad -R^{-1}<\alpha<0,
\end{equation}
with equality if and only if $\Omega=B$. This is precisely the proof of Theorem A for $\alpha\in (-R^{-1}, 0)$ in  \cite{Laugesen=AJM}. The proof of Theorem A for $\alpha\in \left(-\frac{n+1}{n}R^{-1},-R^{-1}\right]$ is different and we refer the interested readers to \cite{Laugesen=AJM} for details. \vspace{0.5em}

Now, we can establish the stability estimate for \eqref{target} in the spirit of \cite{Brasco=GAFA}. 
\begin{proof}[Proof of Theorem \ref{main result one}]
	Since one naturally has $\mathcal{A}(\Omega)\leq 2$, if $\lambda_2(\Omega;\alpha)\leq0$, then Theorem \ref{main result one} is trivial by taking $\gamma(n,\alpha,R)=\lambda_2(B;\alpha)/4>0$. Hence we only consider the case $\lambda_2(\Omega;\alpha)>0$. Our starting point is to rewrite the inequality \eqref{use} as
	\begin{equation}\label{use2}
		\lambda_2(B;\alpha)\int_{B}g(r)^2\,dx-\lambda_2(\Omega;\alpha)\int_{\Omega}g(r)^2\,dx\geq\int_{B} h(r)\,dx-\int_{\Omega} h(r)\,dx.
	\end{equation}
	Our goal is to refine this inequality. By the definition of the Fraenkel asymmetry, one has
	$$
	\mathcal{A}(\Omega)\leq\beta\coloneqq\frac{|\Omega\Delta B|}{|\Omega|}\leq 2.
	$$
	Let $B_1,$ $B_2$ be two balls, concentric with $B$, with radii $R_1\leq R\leq R_2$, such that
$$
	|\Omega\cap B|=\omega_n R_1^n=|B_1|\quad\text{and}\quad|\Omega\backslash B|=|B_2\backslash B|=\omega_n(R_2^n-R^n).
	$$
	Therefore, it follows from $|\Omega|=|B|$ that
	$$
	\frac{R^n-R_1^n}{R^n}=\frac{\beta}{2}=\frac{R_2^n-R^n}{R^n}.
	$$
	On the one hand, by \eqref{gr2} and \eqref{g}, we have
	$$
	\begin{aligned}
		\int_{\Omega}g(r)^2\,dx&\geq\int_{B}g(r)^2\,dx=\int_{B}r^{2-n}J_{n/2}\left(\sqrt{\lambda_2(\mathbb{B};R\alpha)}\frac{r}{R}\right)^2\,dx\\
		&=R^2\int_{B_1(0)}r^{2-n}J_{n/2}\left(\sqrt{\lambda_2(\mathbb{B};R\alpha)}r\right)^2\,dx\\
		&=|\Omega|^{2/n} \eta(n,\alpha,R),
	\end{aligned}
	$$
	where
	$$
	\eta(n,\alpha,R)\coloneqq\omega_n^{-2/n}\int_{B_1(0)}r^{2-n}J_{n/2}\left(\sqrt{\lambda_2(\mathbb{B};R\alpha)}r\right)^2\,dx
	$$
	is a positive constant which only depends on the dimension $n$, the volume of $\Omega$ and the Robin parameter $\alpha$. Therefore, we have
	\begin{equation}\label{final1}
		\lambda_2(B;\alpha)\int_{B}g(r)^2\,dx-\lambda_2(\Omega;\alpha)\int_{\Omega}g(r)^2\,dx\leq|\Omega|^{2/n} \eta(n,\alpha,R) \left(\lambda_2(B;\alpha)-\lambda_2(\Omega;\alpha)\right).
	\end{equation}
	On the other hand, since $h(r)$ is strictly decreasing in $(0,\infty)$, a similar argument as \eqref{gr2} implies
	$$
	\int_{\Omega}h(r)\,dx=\int_{\Omega\cap B}h(r)\,dx+\int_{\Omega\backslash B}h(r)\,dx\leq\int_{B_1}h(r)\,dx+\int_{B_2\backslash B}h(r)\,dx.
	$$
	Hence again by the monotone property of $h(r)$, we have
	\begin{equation}\label{use22}
		\begin{aligned}
			\int_{B}h(r)\,dx-\int_{\Omega}h(r)\,dx&\geq \int_{B\backslash B_1}h(r)\,dx-\int_{B_2\backslash B}h(r)\,dx\\
			&\geq\int_{B\backslash B_1}h(R)\,dx-\int_{B_2\backslash B}h(r)\,dx\\
			&=\int_{B_2\backslash B}\left(h(R)-h(r)\right)\,dx\\
			&=n\omega_n\int_{R}^{R_2}\left(h(R)-h(r)\right)r^{n-1}\,dr.
		\end{aligned}
	\end{equation}
	where we have used the fact $|B\backslash B_1|=|B_2\backslash B|$. Notice from \eqref{g extend} that
	$$
	\begin{aligned}
		h(r)&= g^\prime(r)^2+(n-1)r^{-2}g(r)^2+2\alpha g(r)g^\prime(r)+\alpha\frac{n-1}{r}g(r)^2\\
		&=g(R)^2\left(-\alpha^2+(n-1)\frac{1+\alpha r}{r^2}\right)e^{-2\alpha(r-R)},\quad\forall r\geq R.
	\end{aligned}
	$$
	The next lemma is crucial for the proof.
\begin{lemma}\label{keypoint}
	For $-R^{-1}<\alpha<0$, there holds
	\begin{equation}\label{keypoint=equation}
		\int_{R}^{R_2}\left(h(R)-h(r)\right)r^{n-1}\,dr\geq\frac{7(n-1)R^{n-2}}{16^2n^2}g(R)^2\beta^2.
	\end{equation}
\end{lemma}
\begin{proof}[Proof of Lemma \ref{keypoint}]
	Let $w(r)\coloneqq h(R)-h(r)$ for $r\geq R$. By a direct derivative calculation, we have
	$$
	\begin{aligned}
		w^\prime(r)&=-h^\prime(r)=-g(R)^2\left(2\alpha^3-\frac{2(n-1)}{r^3}\left(\alpha^2r^2+\frac{3}{2}\alpha r+1\right)\right)e^{-2\alpha(r-R)}\\
		&=g(R)^2\left(2|\alpha|^3+\frac{2(n-1)}{r^3}\left[\left(\alpha r+\frac{3}{4}\right)^2+\frac{7}{16}\right]\right)e^{2|\alpha|(r-R)}\\
		&\geq \frac{7(n-1)}{8}g(R)^2\frac{1}{r^3},\quad \forall\,r\geq R,
	\end{aligned}
	$$
	from which we know
	$$
	\begin{aligned}
		h(R)-h(r)&=w(r)-w(R)=\int_R^r w^\prime(s)\,ds\\
		&\geq \frac{7(n-1)}{8}g(R)^2\int_R^r  \frac{1}{s^3}\,ds\\
		&=\frac{7(n-1)}{16}g(R)^2\left(\frac{1}{R^2}-\frac{1}{r^2}\right),\quad \forall\,r\geq R.
	\end{aligned}
	$$
	Therefore, we have
	\begin{equation}\nonumber
		\begin{aligned}
			I&\coloneqq\int_{R}^{R_2}\left(h(R)-h(r)\right)r^{n-1}\,dr\geq \frac{7(n-1)}{16}g(R)^2\int_R^{R_2}\left(\frac{1}{R^2}-\frac{1}{r^2}\right)r^{n-1}\,dr\\
			&=\begin{cases}
				\frac{7(n-1)}{16}g(R)^2\left[\frac{1}{nR^2}\left(R_2^n-R^n\right)-\frac{1}{n-2}\left(R_2^{n-2}-R^{n-2}\right)\right] \quad &\hbox{if}\ n\geq 3, \\
					\frac{7(n-1)}{16}g(R)^2\left[\frac{1}{2R^2}\left(R_2^2-R^2\right)-\log \frac{R_2}{R}\right] \quad &\hbox{if}\ n=2. \\
			\end{cases}
		\end{aligned}
	\end{equation}
Before proceeding any further, one should keep in mind that
$$
R_2=R\left(1+\frac{\beta}{2}\right)^{1/n}\quad\text{and}\quad R_2^n-R^n=R^n\frac{\beta}{2}.
$$
For $n\geq3$, by $0\leq\frac{\beta}{2}\leq 1$ and the elementary inequality
$$
(1+x)^\frac{n-2}{n}\leq 1+\frac{n-2}{n}x-\frac{n-2}{4n^2}x^2,\quad\forall\,x\in [0,1],
$$
we have
$$
\begin{aligned}
	R_2^{n-2}-R^{n-2}&=R^{n-2}\left(\left(1+\frac{\beta}{2}\right)^{\frac{n-2}{n}}-1\right)\\
	&\leq R^{n-2}\left(\frac{n-2}{2n}\beta-\frac{n-2}{16n^2}\beta^2\right).
\end{aligned}
$$
Then, one obtains that
$$
\begin{aligned}
	I&\geq \frac{7(n-1)}{16}g(R)^2\left[\frac{1}{nR^2}\left(R_2^n-R^n\right)-\frac{1}{n-2}\left(R_2^{n-2}-R^{n-2}\right)\right]\\
	&\geq \frac{7(n-1)}{16}g(R)^2\left[\frac{1}{nR^2}R^n\frac{\beta}{2}-\frac{1}{n-2}R^{n-2}\left(\frac{n-2}{2n}\beta-\frac{n-2}{16n^2}\beta^2\right)\right]\\
	&= \frac{7(n-1)}{16^2n^2}g(R)^2 R^{n-2}\beta^2.
\end{aligned}
$$
Similarly, for $n=2$, by $0\leq\frac{\beta}{2}\leq 1$ and the elementary inequality
$$
\log(1+x)\leq x-\frac{1}{8}x^2,\quad\forall\,x\in [0,1],
$$
we have
$$
\log\frac{R_2}{R}=\frac{1}{2}\log\left(1+\frac{\beta}{2}\right)\leq \frac{1}{2}\left(\frac{\beta}{2}-\frac{1}{32}\beta^2\right)=\frac{\beta}{4}-\frac{1}{64}\beta^2,
$$
from which we get
$$
\begin{aligned}
	I&\geq \frac{7(n-1)}{16}g(R)^2\left[\frac{1}{2R^2}\left(R_2^2-R^2\right)-\log \frac{R_2}{R}\right]\\
	&\geq \frac{7(n-1)}{16}g(R)^2\left[\frac{\beta}{4}-\left(\frac{\beta}{4}-\frac{1}{64}\beta^2\right)\right]\\
	&= \frac{7(n-1)}{16^2n^2}g(R)^2 R^{n-2}\beta^2.
\end{aligned}
$$
The proof is complete.
\end{proof}
Now we come back to the proof of Theorem \ref{main result one}. By \eqref{use2}, \eqref{final1}, \eqref{use22} and Lemma \ref{keypoint}, we have proved that
	\begin{equation}\label{final}
		\begin{aligned}
			&\quad|\Omega|^{2/n} \eta(n,\alpha,R) \left(\lambda_2(B;\alpha)-\lambda_2(\Omega;\alpha)\right)\\
			&\geq \omega_n\frac{7(n-1)R^{n-2}}{16^2n}g(R)^2\beta^2= \omega_nJ_{n/2}(\sqrt{\lambda_2(\mathbb{B};R\alpha)})^2\frac{7(n-1)}{16^2n}\beta^2.
		\end{aligned}
	\end{equation}
	Therefore, when $-R^{-1}< \alpha<0$, one obtains
	\begin{equation}\label{conclusion}
		\begin{aligned}
			&\quad \lambda_2(B;\alpha)-\lambda_2(\Omega;\alpha)\\
			&\geq\eta(n,\alpha,R)^{-1}|\Omega|^{-2/n}\omega_nJ_{n/2}(\sqrt{\lambda_2(\mathbb{B};R\alpha)})^2\frac{7(n-1)}{16^2n}\mathcal{A}(\Omega)^2\\
			&\eqqcolon \gamma(n,\alpha,R)\mathcal{A}(\Omega)^2.
		\end{aligned}
	\end{equation}
The proof is complete.
\end{proof}
\begin{proof}[Proof of Corollary \ref{main result two}]
Recall that the Robin eigenvalue $\lambda_k(\Omega;\alpha)$ is continuous as a function of the boundary parameter $\alpha$, so
$$
\mu_2(\Omega)=\lambda_2(\Omega;0)=\lim_{\alpha\to 0}\lambda_2(\Omega;\alpha).
$$
Then by letting $\alpha\to 0^-$ in \eqref{conclusion} and using
$$
\eta(n,\alpha,R)=\omega_n^{-2/n}\int_{B_1(0)}r^{2-n}J_{n/2}\left(\sqrt{\lambda_2(\mathbb{B};R\alpha)}r\right)^2\,dx,
$$
we conclude from the continuity property of the Bessel function that
$$
\mu_2(B)-\mu_2(\Omega)\geq \delta(n)|\Omega|^{-2/n}\mathcal{A}(\Omega)^2,
$$
where
$$
\delta(n):=\frac{7(n-1)}{16^2n}\omega_n^{(n+2)/n}J_{n/2}(\sqrt{\mu_2(\mathbb{B}}))^2\left(\int_{B_1(0)}r^{2-n}J_{n/2}\left(\sqrt{\mu_2(\mathbb{B})}r\right)^2\,dx\right)^{-1}
$$
is a positive constant depending only on $n$ since $\mu_2(\mathbb{B})=\xi^2_{n/2,1}$, where $\xi_{n/2,1}$ denotes the first positive zero of the derivative of $t\mapsto t^{1-n/2}J_{n/2}(t)$. This proves the quantitative Szeg\"o-Weinberger inequality \eqref{main result two inequality}.
\end{proof}

\section{Sharpness of the quantitative Freitas-Laugesen inequality}
In this section, we will show that the exponent 2 in the quantitative Freitas-Laugesen inequality \eqref{main result one inequality} is sharp. Since the second Robin eigenvalue for a ball is not simple, one can not use the usual ellipsoids family to check the sharpness of the exponent 2. Here, we mainly follow the approach from \cite{Brasco=Jfa,Brasco=GAFA} to construct a family of nearly spherical domains to verify the sharpness of exponent 2. Such domains were used to check the sharpness of exponent 2 in the quantitative Szeg\"o-Weinberger inequality \cite{Brasco=GAFA} and quantitative Brock-Weinstock inequality \cite{Brasco=Jfa}.\vspace{0.5em} 

For our problem, compared to the Neumann case \cite{Brasco=GAFA} and Steklov case \cite{Brasco=Jfa}, since the first Robin eigenfunction is not constant, some extra difficulties will arise. Therefore, the nearly spherical domains we construct later need some extra symmetric property. Since the whole proof is quite complicated, we will divide this section into various subsections to facilitate better readability. In the sequel, we will use $K, C, C_1,\cdots$ to denote positive constants that are  independent of $\e$ but may be different in different places.

\subsection{Setting of the construction and basic properties}
Throughout this section, let $B\subset\Rn$ be the unit ball centered at the origin and $-1<\alpha<0$. As in \cite{Brasco=Jfa,Brasco=GAFA}, we consider a general nearly spherical domain given by
\begin{equation}\label{omega epsilon}
	\Omega_\varepsilon\coloneqq\{x\in\Rn:x=0\text{ or }|x|<1+\varepsilon\psi(x/|x|)\},\quad0<\varepsilon\ll 1,
\end{equation}
where $\psi\in C^\infty(\partial B)$ is a smooth function on the unit sphere $\partial B$. The function $\psi$ is selected to satisfy the following assumptions, which are crucial for the proof.
\begin{equation}\label{assumption1}
	\int_{\partial B}\psi(x)\,d\mathcal{H}^{n-1}=0,
\end{equation}
\begin{equation}\label{assumption2}
	\int_{\partial B}\langle a,x\rangle\psi(x)\,d\mathcal{H}^{n-1}=0,\quad\text{for every }a\in\Rn,
\end{equation}
\begin{equation}\label{assumption3}
	\int_{\partial B}\langle a,x\rangle^2\psi(x)\,d\mathcal{H}^{n-1}=0,\quad\text{for every }a\in\Rn.
\end{equation}
\begin{remark}\label{assumption meaning}
	The assumptions \eqref{assumption1}, \eqref{assumption2} and \eqref{assumption3} are equivalent to require that $\psi$ is orthogonal in the $L^2(\partial B)$ sense to the first three eigenspaces of the Laplace-Beltrami operator on $\partial B$, i.e., to spherical harmonics of order $0$, $1$ and $2$, respectively, from which we know such function $\psi$ does exist. We also refer the interested readers to \cite{Muller=book} for a comprehensive introduction on spherical harmonics.
\end{remark}
\begin{remark}\label{assumption orthogonal}
	By Proposition \ref{second}, we know every Robin eigenfunction $\xi$ relative to $\lambda_2(B;\alpha)$ has the form
	$$
	\xi(x)=g(r)\frac{\langle a,x\rangle}{r}
	$$
	for some $a\in\R^n$. In particular, $\xi(x)=g(1)\langle a,x\rangle$ for $x\in\partial B$. Therefore, it follows from \eqref{assumption1}, \eqref{assumption2} and \eqref{assumption3} that
	\begin{equation}\label{assumption orthogonal equation}
		\int_{\partial B}\psi(x) \xi(x)^2\,d\Hn=0\quad\text{and}\quad\int_{\partial B}\psi(x)|\nabla_{\rm tan}\xi(x)|^2\,d\Hn=0,
	\end{equation}
	where $\nabla_{\rm tan}$ is the tangential gradient. This relation is very crucial for improving the decay rate in later proofs.
\end{remark}
The following geometric result established in \cite{Brasco=Jfa} is the starting point.
\begin{lemma}\cite[Lemma 6.2]{Brasco=Jfa}\label{fraenkel asymmetry}
	Let $\psi\in C^\infty(\partial B)$ satisfy \eqref{assumption1} and \eqref{assumption2}, then
	\begin{equation}
		|\Omega_\varepsilon|-|B|\simeq \varepsilon^2\quad\text{and}\quad\mathcal{A}(\Omega_\varepsilon)\simeq\varepsilon\simeq\frac{|\Omega_\varepsilon\Delta B|}{|\Omega_\varepsilon|}.
	\end{equation}
	As a direct corollary, there exists a positive constant $K>0$ independent of $\varepsilon$, such that
	\begin{equation}\label{volume estimate}
		|\Omega_\varepsilon\backslash B|\leq K\varepsilon\quad\text{and}\quad |B\backslash\Omega_\varepsilon|\leq K\varepsilon.
	\end{equation}
\end{lemma}
One can see later that we will use the extension of the second Robin eigenfunction on $\Omega_\varepsilon$ as a test function for $\lambda_2(B;\alpha)$. This is the most different part from \cite{Brasco=Jfa,Brasco=GAFA} since we do not know the connection between the first Robin eigenfunctions on $\Omega_\varepsilon$ and $B$, while in the Neumann case  \cite{Brasco=GAFA} and Steklov case  \cite{Brasco=Jfa}, the first eigenfunctions on $\Omega_\e$ and $B$ are both simply constants. For this reason, we need the following extra assumption on $\psi$.\vspace{0.5em}

 In what follows, except for conditions \eqref{assumption1}, \eqref{assumption2} and \eqref{assumption3}, we also assume that the domains $\{\Omega_\varepsilon\}_{\varepsilon>0}$ are symmetric in $x_i$-direction for all $i=1,\cdots,n$, i.e., if $(x_1,\cdots,x_i,\cdots,x_n)\in\Omega_\e$, then $(x_1,\cdots,-x_i,\cdots,x_n)\in\Omega_\e$. These four assumptions are compatible in the sense that such $\psi$ does exist. For instance, when $n=2$, we can take
$$
\psi:[0,2\pi]\to\R,\quad \theta\mapsto \cos 4\theta.
$$
The following lemma concerning the anti-symmetry of some second Robin eigenfunction is the keypoint for the integral estimate throughout this section.
\begin{lemma}\label{antisymmetry}
	Assume that $\Omega\subset\Rn$ is a smooth bounded domain which is symmetric in $x_i$-direction for all $i=1,\cdots,n$, i.e., $(x_1,\cdots,x_i,\cdots,x_n)\in\Omega$ implies $(x_1,\cdots,-x_i,\cdots,x_n)\in\Omega$. Then there exists $i_0\in\mathbb{N}_+$, $1\leq i_0\leq n$, such that there exists a non-trivial eigenfunction relative to $\lambda_2(\Omega;\alpha)$ which is anti-symmetric in $x_{i_0}$-direction, i.e.,
	$$u(x_1,\cdots,-x_{i_0},\cdots, x_n)=-u(x_1,\cdots,x_{i_0},\cdots, x_n).$$
\end{lemma}
\begin{proof}
	Let $u_2\neq 0$ be a smooth second Robin eigenfunction on $\Omega$, i.e.,
	$$
	\begin{cases}
		-\Delta u_2=\lambda_2(\Omega;\alpha) u_2 \quad &\hbox{in}\ \Omega, \\
		\frac{\partial u_2}{\partial \nu}+\alpha u_2=0 \quad &\hbox{on}\ \partial\Omega. \\
	\end{cases}
	$$
	If $u_2$ is not symmetric in $x_{i_0}$-direction for some $1\leq i_0\leq n$, we can take
	$$
	v(x)=u_2(x_1,\cdots,x_{i_0},\cdots,x_n)-u_2(x_1,\cdots,-x_{i_0},\cdots,x_n),\quad x\in\Omega,
	$$
	then one can easily verify that $v$ is a non-trivial eigenfunction for $\lambda_2(\Omega;\alpha)$ and anti-symmetric in $x_{i_0}$-direction. So it suffices to prove that $u_2$ can not be symmetric in every $x_i$-direction.
	
Assume by contradiction that $u_2$ is symmetric in $x_i$-direction for all $1\leq i\leq n$, that is, $u_2(x_1,\cdots,x_{i},\cdots,x_n)=u_2(x_1,\cdots,-x_{i},\cdots,x_n)$ for any $x\in\Omega$. Courant's Nodal Domain theorem asserts that $u_2$ has exactly $2$ nodal domains $\Omega_\pm\coloneqq\{x\in\Omega:\operatorname{sgn}(u_2(x))=\pm 1\}$ since the first eigenfunction has constant sign. On the other hand, by denoting $\lambda_1^D(\Omega)$ the first Dirichlet eigenvalue on $\Omega$, it is known (cf. \cite{Filonov=comparison result} or \cite[page 123]{Brasco=GAFA}) that $\mu_2(\Omega)<\lambda_1^D(\Omega)$, 
so it follows from the variational principle \eqref{variation2} and $\alpha<0$ that
	$$
	\lambda_2(\Omega;\alpha)\leq\lambda_2(\Omega;0)=\mu_2(\Omega)<\lambda^D_1(\Omega),
	$$
Hence, the nodal set $\overline{\{x\in\Omega: u_2(x)=0\}}$ of $u_2$ must touch the boundary $\partial \Omega$. Otherwise, there exists $\Omega^\prime\subset\subset\Omega$ such that $u_2|_{\partial\Omega^\prime}\equiv0$. In fact, one can take $\Omega^\prime$ as an inner nodal domain of $u_2$. Then 
$$-\Delta u_2=\lambda_2(\Omega; \alpha)u_2\quad\text{in }\Omega',\quad u_2=0\quad\text{on }\partial\Omega',$$
so
	$$
	\lambda_1^D(\Omega^\prime)\leq \lambda_2(\Omega;\alpha)<\lambda^D_1(\Omega),
	$$
	which obviously contradicts to the fact $\lambda^D_1(\Omega)<\lambda^D_1(\Omega')$ because of $\Omega^\prime\subset\subset\Omega$. 
	
	Therefore, the nodal set $\overline{\{x\in\Omega: u_2(x)=0\}}$ of $u_2$ touches the boundary $\partial\Omega$.
	Let
	$$
	\Omega_1\coloneqq\{x\in\Omega:x_i>0,i=1,\cdots,n\},
	$$
	which is not empty by the symmetry of $\Omega$. Also, $\Omega$ can be viewed as a $2^n$ times copy of $\Omega_1$ in some sense. Since $u_2$ is symmetric with respect to $x_i$-direction for all $1\leq i\leq n$, we know both $\Omega_+\cap\Omega_1$ and $\Omega_-\cap\Omega_1$ are not empty. Then by the connectivity of $\Omega_+$ and $\Omega_-$, we know the nodal set can never touch the boundary, hence a contradiction again.
\end{proof}
Now, let us fix an eigenfunction $u_\varepsilon$ for $\lambda_2(\Omega_\varepsilon;\alpha)$, normalized in such a way that
\begin{equation}\label{normalized}
	\int_{\Oe}\ue(x)^2\,dx=1\quad\text{and}\quad\int_{\Oe}|\nabla \ue(x)|^2\,dx+\alpha\int_{\partial \Oe}\ue(x)^2\,d\Hn=\lambda_2(\Oe;\alpha).
\end{equation}
By Lemma \ref{antisymmetry}, we can assume without loss of generality that $\ue$ is anti-symmetric in $x_1$-direction, i.e., $$\ue(-x_1,x_2,\cdots,x_n)=-\ue(x_1,x_2,\cdots,x_n).$$
\begin{remark}[\cite{Brasco=Jfa, Brasco=GAFA}]\label{regularity estimate}
	Thanks to the fact that $\partial\Oe$ is of class $C^\infty$, we obtain that $\ue\in C^\infty(\overline{\Oe})$. Moreover, the domains $\Oe$ are uniformly of class $C^k$ for every $k\in\mathbb{N}$, hence we can assume the functions $\ue$ to satisfy uniform $C^k$ estimates, i.e.,
	\begin{equation}\label{regularity estimate equation}
		\|\ue\|_{C^k(\overline{\Oe})}\leq H_k,
	\end{equation}
	for some positive constants $H_k$ depending only on $k\in\mathbb{N}$. 
\end{remark}
Now, we consider a $C^{k_0}$ extension ${\widetilde{u}_{\varepsilon}}$ of $\ue$ to a larger set
\begin{equation}
	D_\varepsilon\coloneqq\{x\in\Rn:|x|\leq 1+\varepsilon\|\psi\|_{L^\infty(\partial B)}\}\supset\overline{B\cup\Oe}
\end{equation}
in such a way that
\begin{equation}\label{regularity estimate equation extension}
	\|{\widetilde{u}_{\varepsilon}}\|_{C^{k_0}(D_\varepsilon)}\leq K\|\ue\|_{C^{k_0}(\overline{\Oe})}
\end{equation}
for some positive constant $K$ independent of $\varepsilon$, see also \cite{Brasco=Jfa,Brasco=GAFA} for such kind of extensions. The value of ${k_0}\in \mathbb{N}_+$ will be given later. We will use this extension to construct a test function for $\lambda_2(B;\alpha)$ and then yield the following estimate.
\begin{lemma}\label{first estimate}
	Let $0<\varepsilon_0\ll1$, then there exist three functions $R_1,R_2,R_3:[0,\e_0]\to\R$ and a constant $C>0$ such that for every $0<\e<\e_0$, we have
	\begin{equation}\label{estimate for lambda2B}
		0<\lambda_2(B;\alpha)\leq\frac{\lambda_2(\Oe;\alpha)+R_1(\e)+R_2(\e)+C\e^2}{1+R_3(\e)-C\e^2}.
	\end{equation}
	Moreover, we have
	\begin{equation}\label{estimate for R123}
		|R_i(\varepsilon)|\leq C\varepsilon,\quad i=1,2,3,
	\end{equation}
	possibly with a different constant $C>0$ which is also independent of $\e$. Consequently, $\lambda_2(\Oe;\alpha)>0$ for $\e>0$ small.
\end{lemma}
\begin{proof}
	Let $u_1>0$ be the first Robin eigenfunction for $\lambda_1(B;\alpha)$ such that $\int_B u_1(x)\,dx=1$. From Proposition \ref{first}, we know $u_1$ is radially symmetric in $B$. We set
	$$
	\delta\coloneqq\int_{B}{\widetilde{u}_{\varepsilon}}(x)u_1(x)\,dx.
	$$
	By Lemma \ref{fraenkel asymmetry} and the uniform estimates \eqref{regularity estimate equation} and \eqref{regularity estimate equation extension}, we know
	$$
	\begin{aligned}
		\delta&=\int_{B\cap\Oe}{\widetilde{u}_{\varepsilon}}(x)u_1(x)\,dx+\int_{B\backslash\Oe}{\widetilde{u}_{\varepsilon}}(x)u_1(x)\,dx\\
		&=\int_{B\cap\Oe}{\widetilde{u}_{\varepsilon}}(x)u_1(x)\,dx+O(\e).
	\end{aligned}
	$$
	Since $\Oe$ is symmetric in every $x_i$-direction, $u_1$ is radially symmetric and $\ue$ is anti-symmetric in $x_1$-direction, we know
	$$
	\int_{B\cap\Oe}{\widetilde{u}_{\varepsilon}}(x)u_1(x)\,dx=0
	$$
	and hence $\delta=O(\varepsilon)$. Now, we are ready to define the test function for $\lambda_2(B;\alpha)$ as follows
	\begin{equation}\label{ve}
		v_\e(x)\coloneqq{\widetilde{u}_{\varepsilon}}(x)\cdot\mathbf{1}_{\overline{B}}(x)-\delta.
	\end{equation}
	Obviously, by the definition of $\delta$, we know $v_\e\perp u_1$ in the $L^2(B)$ sense, i.e.,
	\begin{equation}\label{veu1}\int_{B} v_\e u_1dx=0.\end{equation}
	 Also, by \eqref{regularity estimate equation extension} and the fact that $\delta=O(\e)$, we know
	\begin{equation}\label{regularity estimate v}
		\|v_\e\|_{C^{k_0}(\overline{B})}\leq C
	\end{equation}
	for some constant $C>0$ independent of $\e$. Firstly, notice that
	$$
	\begin{aligned}
		\int_{B}|\nabla v_\e(x)|^2\,dx&=\int_{\Oe}|\nabla \ue(x)|^2\,dx+\int_{B\backslash\Oe}|\nabla v_\e(x)|^2\,dx-\int_{\Oe\backslash B}|\nabla \ue(x)|^2\,dx\\
		&\eqqcolon\int_{\Oe}|\nabla \ue(x)|^2\,dx+R_1(\e).
	\end{aligned}
	$$
	To estimate the boundary integral, we define $\phi_\e:\partial B\to\partial \Oe$ by
	\begin{equation}\label{boundary map}
		\phi_\e(x)\coloneqq x+\varepsilon\psi(x)x,\quad x\in\partial B.
	\end{equation}
	Notice that 
	$$
	{\widetilde{u}_{\varepsilon}}(\phi_\e(x))={\ue}(\phi_\e(x)),\quad x\in\partial B.
	$$
	On the other hand, by the uniform estimates \eqref{regularity estimate equation} and \eqref{regularity estimate equation extension}, we have
	\begin{equation}\label{boundary error}
		{\widetilde{u}_{\varepsilon}}(x)=\ue(\phi_\e(x))+O(\e),\quad x\in\partial B.
	\end{equation}
	Therefore, 
	$$
	\begin{aligned}
		\left|\int_{\partial B}v_\e(x)^2\,d\Hn-\int_{\partial B}{\widetilde{u}_{\varepsilon}}(x)^2\,d\Hn\right|&=\left|\delta^2\Hn(\partial B)-2\delta\int_{\partial B}{\widetilde{u}_{\varepsilon}}(x)\,d\Hn\right|\\
		&\leq2|\delta|\left|\int_{\partial B}{\widetilde{u}_{\varepsilon}}(x)\,d\Hn\right|+O(\e^2)\\
		&=2|\delta|\left|\int_{\partial B}\ue(\phi_\e(x))\,d\Hn+O(\e)\right|+O(\e^2)\\
		&=2|\delta|\left|\int_{\partial B}\ue(\phi_\e(x))\,d\Hn\right|+O(\e^2).
	\end{aligned}
	$$
	Define the Jacobian 
	\begin{equation}\label{jacobian}
			J_\e(x):=(1+\e\psi(x))^{n-2}\sqrt{(1+\varepsilon\psi(x))^2+\e^2|\nabla_{\rm tan}\psi(x)|^2},\quad x\in\partial B.
	\end{equation}
	By a straightforward computation, we have
	\begin{equation}\label{error Je}
		\|J_\e-1\|_{L^\infty(\partial B)}=O(\e), 
	\end{equation}
	from which and the change of variable $y=\phi_\e(x)$ we know
	$$
	\begin{aligned}
		\int_{\partial B}\ue(\phi_\e(x))\,d\Hn&=\int_{\partial B}\ue(\phi_\e(x))J_\e(x)\,d\Hn+O(\e)\\
		&=\int_{\partial \Oe}\ue(y)\,d\Hn+O(\e)
		=O(\e),
	\end{aligned}
	$$
	where we have used $\int_{\partial \Oe}\ue(y)\,d\Hn=0$ due to the anti-symmetry in $x_1$-direction of $\ue$. Hence, we know
	\begin{equation}\label{youyong2}
		\int_{\partial B}v_\e(x)^2\,d\Hn=\int_{\partial B}{\widetilde{u}_{\varepsilon}}(x)^2\,d\Hn+O(\e^2).
	\end{equation}
	Now we define
	$$
	R_2(\e)\coloneqq\alpha\left(\int_{\partial B}{\widetilde{u}_{\varepsilon}}(x)^2\,d\Hn-\int_{\partial \Oe}\ue(x)^2\,d\Hn\right),
	$$
then
	$$
	\alpha\int_{\partial B}v_\e(x)^2\,d\Hn=\alpha\int_{\partial \Oe}\ue(x)^2\,d\Hn+R_2(\e)+O(\e^2).
	$$
Finally, for the $L^2$ integral on $B$, we know
    \begin{equation}\label{youyong}
    	\begin{aligned}
    		\left|\int_{B}v_\e(x)^2\,dx-\int_{B}{\widetilde{u}_{\varepsilon}}(x)^2\,dx\right|&=\left|\delta^2 |B|-2\delta\int_{B}{\widetilde{u}_{\varepsilon}}(x)\,dx\right|\\
    		&\leq 2|\delta|\left|\int_{B\cap\Oe}{\ue}(x)\,dx+\int_{B\backslash\Oe}{\widetilde{u}_{\varepsilon}}(x)\,dx\right|+C\varepsilon^2\\
    		&\leq C\e^2,
    	\end{aligned}
    \end{equation}
	where we have used $\int_{B\cap\Oe}{\ue}(x)\,dx=0$ due to the anti-symmetry in $x_1$-direction of $\ue$, the volume estimate \eqref{volume estimate} and the uniform estimates \eqref{regularity estimate equation}, \eqref{regularity estimate equation extension}. Therefore, it follows from $\widetilde{u}_{\varepsilon}(x)=\ue(x)$ for $x\in \overline{\Oe}$ that
	$$
	\begin{aligned}
		\int_{B}v_\e(x)^2\,dx&\geq \int_{B}{\widetilde{u}_{\varepsilon}}(x)^2\,dx-C\e^2\\
		&=\int_{\Oe}\ue(x)^2\,dx+\int_{B\backslash\Oe}{\widetilde{u}_{\varepsilon}}(x)^2\,dx- \int_{\Oe\backslash B}\ue(x)^2\,dx-C\e^2\\
		&\eqqcolon \int_{\Oe}\ue(x)^2\,dx+R_3(\e)-C\e^2.
	\end{aligned}
	$$
	Together with these three estimates, \eqref{veu1} and \eqref{normalized}, we are now able to estimate $\lambda_2(B;\alpha)$ as follows
	$$
	\begin{aligned}
		\lambda_2(B;\alpha)&\leq\frac{\int_{B}|\nabla v_\e(x)|^2\,dx+\alpha\int_{\partial B}v_\e(x)^2\,d\Hn}{\int_{B}v_\e(x)^2\,dx}\\
		&\leq \frac{\lambda_2(\Oe;\alpha)+R_1(\e)+R_2(\e)+C\e^2}{1+R_3(\e)-C\e^2}.
	\end{aligned}
	$$
	Now we prove the estimates \eqref{estimate for R123}. Recall that
		$$
	\begin{aligned}
		R_1(\e)=\int_{B\backslash\Oe}|\nabla v_\e(x)|^2\,dx-\int_{\Oe\backslash B}|\nabla \ue(x)|^2\,dx,
	\end{aligned}
	$$
	$$
	R_2(\e)=\alpha\left(\int_{\partial B}{\widetilde{u}_{\varepsilon}}(x)^2\,d\Hn-\int_{\partial \Oe}\ue(x)^2\,d\Hn\right),
	$$
	$$
	R_3(\e)=\int_{B\backslash\Oe}{\widetilde{u}_{\varepsilon}}(x)^2\,dx- \int_{\Oe\backslash B}\ue(x)^2\,dx.
	$$
	By the uniform estimates \eqref{regularity estimate equation}, \eqref{regularity estimate equation extension} and the volume estimates \eqref{volume estimate}, it is easy to know that $|R_1(\e)|\leq C\e$ and $|R_3(\e)|\leq C\e$. For the second term, we have
	$$
	\begin{aligned}
		\int_{\partial\Oe}{\ue}(x)^2\,d\Hn&=\int_{\partial B}\ue(\phi_\e(x))^2J_\e(x)\,d\Hn\\
		&=\int_{\partial B}{\widetilde{u}_{\varepsilon}}(\phi_\e(x))^2J_\e(x)\,d\Hn\\
		&=\int_{\partial B}({\widetilde{u}_{\varepsilon}}(x)+O(\e))^2(1+O(\e))\,d\Hn\\
		&=\int_{\partial B}{\widetilde{u}_{\varepsilon}}(x)^2\,d\Hn+O(\e).
	\end{aligned}
	$$
	The proof is complete.
\end{proof}
Finally, in this subsection, we give an upper bound for $|\lambda_2(B;\alpha)-\lambda_2(\Omega;\alpha)|$, which is useful for approximating $v_\e$ by the second Robin eigenfunction relative to $\lambda_2(B;\alpha)$ in the Sobolev $W^{1,2}$ sense. 
\begin{lemma}\label{upper bound}
	There exists a positive constant $C$, such that
	$$
	|\lambda_2(B;\alpha)-\lambda_2(\Omega_\e;\alpha)|\leq C(|R_1(\e)|+|R_2(\e)|+|R_3(\e)|)+C\e^2,\quad\text{for every }0<\e\ll1.
	$$
\end{lemma}
Different from the Neumann problem and Steklov problem, since the Robin problem is not scaling invariant, it is not so obvious for the validity of Lemma \ref{upper bound} and we give a detailed proof here.
\begin{proof}
	For $0<\e\ll1$, by Lemma \ref{first estimate} and the Freitas-Laugesen inequality \eqref{Laugesen}, we know
	$$
	0<\lambda_2(\Oe;\alpha)\leq \lambda_2(B_\e;\alpha)=t_\e^{-2}\lambda_2(B;t_\e \alpha),
	$$
	where $B_\e$ is a ball center at the origin such that $|B_\e|=|\Oe|$ and 
	$$
	t_\e=\left(\frac{|\Omega_\e|}{|B|}\right)^{1/n}.
	$$
	Again, let $u_2$ be a normalized (i.e. $\int_{B}u_2^2dx=1$) eigenfunction for $\lambda_2(B;\alpha)$, then Remark \ref{assumption orthogonal} says that $u_2(-x)=-u_2(x)$. Let $\hat{u}_1$ be the first Robin eigenfunction for $\lambda_1(B;t_\e\alpha)$, then it follows from Proposition \ref{first} that $\hat{u}_1$ is radially symmetric, so $\int_{B} u_2\hat{u}_1dx=0$ and then the variational principle \eqref{variation2} implies that
	$$
	0<\lambda_2(B;t_\e \alpha)\leq Q[u_2; B, t_\e\alpha]= \lambda_2(B;\alpha)+(t_\e-1)\alpha\int_{\partial B}u_2(x)^2\,d\Hn\leq C,
	$$
where we have used Lemma \ref{fraenkel asymmetry} for estimating $t_\e$. Thus $0<\lambda_2(\Oe;\alpha)\leq t_\e^{-2}\lambda_2(B;t_\e \alpha)\leq C$. Then by Lemma \ref{first estimate}, it is not difficult to deduce that
	\begin{equation}\label{diyi}
		\lambda_2(B;\alpha)-\lambda_2(\Omega_\e;\alpha)\leq C(|R_1(\e)|+|R_2(\e)|+|R_3(\e)|)+C\e^2.
	\end{equation}
	For the reversed inequality, notice from the above argument that 
	$$
	\begin{aligned}
		\lambda_2(\Omega_\e;\alpha)-\lambda_2(B;\alpha)&\leq t_\e^{-2}\lambda_2(B;t_\e \alpha)-\lambda_2(B;\alpha)\\
		&\leq (t_\e^{-2}-1)\lambda_2(B;\alpha)+\frac{t_\e-1}{t_\e^2}\alpha\int_{\partial B}u_2(x)^2\,d\Hn.
	\end{aligned}
	$$
	Recall  Lemma \ref{fraenkel asymmetry} that $|\Omega_\varepsilon|-|B|\simeq \varepsilon^2$ and thus
	$$
	t_\e=\left(\frac{|\Omega_\e|}{|B|}\right)^{1/n}=\left(1+O(\e^2)\right)^{1/n}=1+O(\e^2),
	$$
	from which we know
	\begin{equation}\label{dier}
		\lambda_2(\Omega_\e;\alpha)-\lambda_2(B;\alpha)\leq C\e^2.
	\end{equation}
	Combining \eqref{diyi} and \eqref{dier}, we get the desired result.
\end{proof}
\subsection{Improving the decay rate by an eigenfunction approximation}
In this subsection, we will prove that if the distance in $C^1$ between $v_\e$ and the eigenspace corresponding to $\lambda_2(B;\alpha)$ has a certain decay rate, then the decays $|R_i(\e)|$, $i=1,2,3$, can be improved of the same order. This type of result is very important and crucial, and was firstly discovered in \cite{Brasco=Jfa,Brasco=GAFA} for the Neumann problem and Steklov problem.
\begin{lemma}\label{decay rate}
	Let $\omega:[0,1]\to\R_+$ be a continuous function such that 
	$$
	\frac{t}{C}\leq\omega(t)\leq C\sqrt{t}
	$$
	for some positive constant $C>0$ independent of $\e$. Suppose that for every $0<\e\ll1$, there exists an eigenfunction $\xi_\e$ for $\lambda_2(B;\alpha)$ such that
	$$
	\|v_\e-\xi_\e\|_{C^1(\overline{B})}\leq C\omega(\e)
	$$
	for some positive constant $C>0$ independent of $\e$. Then there exists another positive constant  independent of $\e$ and we still denote it $C$, such that
	$$
	|R_1(\e)|+|R_2(\e)|+|R_3(\e)|\leq C\e\omega(\e),\quad\text{for every }0<\e\ll1.
	$$
\end{lemma}
\begin{proof}
The computations are similar to those in \cite{Brasco=Jfa,Brasco=GAFA} but more complicated since there are more terms for the Robin problem. We start to estimate the term $|R_1(\varepsilon)|$. By the uniform estimates \eqref{regularity estimate equation} and the definition \eqref{boundary map} of $\phi_{\varepsilon}$, we have
	\begin{equation}\label{4--17}
	|\nabla u_{\varepsilon}(x)|^2=|\nabla u_{\varepsilon}(\phi_{\varepsilon}(x/|x|))|^2+O(\varepsilon),\quad\forall x \in \Omega_{\varepsilon} \backslash B.
	\end{equation}
	By splitting the gradient in its radial and tangential components, the right-hand side can be written as
	$$
	\begin{aligned}
		|\nabla u_{\varepsilon}(\phi_{\varepsilon}(x /|x|))|^2 & =|\partial_{\rho} u_{\varepsilon}(\phi_{\varepsilon}(x /|x|))|^2+\frac{1}{(1+\varepsilon \psi(x /|x|))^2}|\nabla_{\rm tan} u_{\varepsilon}(\phi_{\varepsilon}(x /|x|))|^2 \\
		& =|\partial_{\rho} u_{\varepsilon}(\phi_{\varepsilon}(x /|x|))|^2+|\nabla_{\rm tan} u_{\varepsilon}(\phi_{\varepsilon}(x /|x|))|^2+O(\varepsilon)
	\end{aligned}
	$$
	with the help of the uniform estimates \eqref{regularity estimate equation}. Since $u_{\varepsilon}$ is an eigenfunction for $\lambda_2(\Oe; \alpha)$, we have
	$$
	\langle\nabla u_{\varepsilon}(x), \nu_{\Omega_{\varepsilon}}(x)\rangle=-\alpha u_{\varepsilon}(x),\quad x\in\partial\Oe,
	$$
	where the outward normal $\nu_{\Omega_{\varepsilon}}$ on $\partial \Oe$ is given by
	$$
\nu_{\Omega_{\varepsilon}}(x)=\frac{(1+\varepsilon \psi(x /|x|)) x /|x|-\varepsilon \nabla_{\rm tan} \psi(x /|x|)}{\sqrt{(1+\varepsilon \psi(x /|x|))^2+|\nabla_{\rm tan} \psi(x /|x|)|^2}}=\frac{x}{|x|}+O(\varepsilon), \quad x \in \partial \Omega_{\varepsilon}.
$$
	Using $\phi_{\varepsilon}(x /|x|)\in\partial\Oe$ and \eqref{regularity estimate equation} once again, we have
	\begin{align}\label{4--18}
	|\nabla u_{\varepsilon}(\phi_{\varepsilon}(x /|x|))|^2&=\alpha^2 u_{\varepsilon}(\phi_{\varepsilon}(x /|x|))^2+|\nabla_{\rm tan} u_{\varepsilon}(\phi_{\varepsilon}(x /|x|))|^2+O(\varepsilon)\nonumber\\
	&=\alpha^2\ue(x/|x|)^2+|\nabla_{\rm tan}\ue(x /|x|)|^2+O(\e).
	\end{align}
	Therefore, by \eqref{omega epsilon}, \eqref{volume estimate}, \eqref{4--17} and \eqref{4--18}, we know
	\begin{equation}\label{yi}
		\begin{aligned}
			\int_{\Oe\backslash B}|\nabla\ue(x)|^2\,dx&=\int_{\partial B\cap\{\psi>0\}}\int_{1}^{1+\varepsilon \psi(x)}|\nabla\ue(sx)|^2ds\,d\Hn\\
			&=\varepsilon\int_{\partial B\cap\{\psi>0\}}\left(\alpha^2\ue(x)^2+|\nabla_{\rm tan}\ue(x )|^2\right)\psi(x)\,d\Hn+O(\e^2)\\
			&=\varepsilon\int_{\partial B\cap\{\psi>0\}}\left(\alpha^2v_\e(x)^2+|\nabla_{\rm tan}v_\e(x )|^2\right)\psi(x)\,d\Hn+O(\e^2),
		\end{aligned}
	\end{equation}
where the last equality comes from the fact that $\partial B\cap\{\psi>0\}\subset\Omega_\e$ and $v_\e=\widetilde{u}_\e-\delta=\ue-\delta=\ue+O(\e)$ in $\overline{\Oe}\cap \overline{B}$. In the very same way, recalling that for $x\in\partial B\backslash\Oe$, one has $\phi_\e(x)\in\partial\Oe\cap \overline{B}$ and so
$$
\nabla v_\e(\phi_\e(x))=\nabla \ue(\phi_\e(x)),\quad x\in\partial B\backslash\Oe.
$$
Then by the uniform estimates \eqref{regularity estimate v} and the same reasoning as before, one obtains for every $x\in B\backslash\Oe$ that
$$
\begin{aligned}
	|\nabla v_{\varepsilon}(x)|^2 &=|\nabla v_{\varepsilon}(\phi_{\varepsilon}(x /|x|))|^2+O(\e)
	=|\nabla u_{\varepsilon}(\phi_{\varepsilon}(x /|x|))|^2+O(\e)\\
	&=\alpha^2\ue(x/|x|)^2+|\nabla_{\rm tan}\ue(x /|x|)|^2+O(\e).
\end{aligned}
$$
Integrating on $B\backslash\Oe$, one similarly has
{\allowdisplaybreaks
		\begin{align}\label{er}
			\int_{B\backslash\Oe}|\nabla v_\e(x)|^2\,dx&=\int_{\partial B\cap\{\psi<0\}}\int_{1+\varepsilon \psi(x)}^{1}|\nabla v_\e(sx)|^2ds\,d\Hn\nonumber\\&=-\varepsilon\int_{\partial B\cap\{\psi<0\}}\left(\alpha^2\ue(x)^2+|\nabla_{\rm tan}\ue(x )|^2\right)\psi(x)\,d\Hn+O(\e^2)\\
			&=-\varepsilon\int_{\partial B\cap\{\psi<0\}}\left(\alpha^2v_\e(x)^2+|\nabla_{\rm tan}v_\e(x )|^2\right)\psi(x)\,d\Hn+O(\e^2).\nonumber
		\end{align}}%
	Combining \eqref{yi} and \eqref{er}, we know from the assumption of this lemma that
	$$
	\begin{aligned}
		|R_1(\e)|&=\left|\int_{B\backslash\Oe}|\nabla v_\e(x)|^2\,dx-\int_{\Oe\backslash B}|\nabla \ue(x)|^2\,dx\right|\\
		&=\left|\varepsilon\int_{\partial B}\left(\alpha^2v_\e(x)^2+|\nabla_{\rm tan}v_\e(x )|^2\right)\psi(x)\,d\Hn+O(\e^2)\right|\\
		&\leq \e\alpha^2\left|\int_{\partial B}\xi_\e(x)^2\psi(x)\,d\Hn\right|+\e\left|\int_{\partial B}|\nabla_{\rm tan}\xi_\e(x )|^2\psi(x)\,d\Hn\right|+C\e\omega(\e)+O(\e^2)\\
		&\leq C\e\omega(\e),
	\end{aligned}
	$$
	where we have used \eqref{assumption orthogonal equation} and $\e=O(\omega(\e))$ for the last inequality. As for the second term $|R_2(\e)|$, notice that
	$$
	\begin{aligned}
		R_2(\e)&=\alpha\left(\int_{\partial B}{\widetilde{u}_{\varepsilon}}(x)^2\,d\Hn-\int_{\partial \Oe}\ue(x)^2\,d\Hn\right)\\
		&=\alpha\left(\int_{\partial B}({\widetilde{u}_{\varepsilon}}(x)^2-\ue(\phi_\e(x))^2J_\e(x))\,d\Hn\right)\\
		&=\alpha\left(\int_{\partial B}({\widetilde{u}_{\varepsilon}}(x)^2-{\widetilde{u}_{\varepsilon}}(\phi_\e(x))^2)\,d\Hn+\int_{\partial B}{\widetilde{u}_{\varepsilon}}(\phi_\e(x))^2(1-J_\e(x))\,d\Hn\right)\\
		&\eqqcolon \alpha(R_{2,1}(\e)+R_{2,2}(\e)).
	\end{aligned}
	$$
	Since $v_\e=\widetilde{u}_\e-\delta=\widetilde{u}_\e+O(\e)$ and $\nabla{\widetilde{u}_{\varepsilon}}(x)=\nabla v_\e(x)$ for $x\in\overline{B}$ by \eqref{ve}, together with the uniform estimates \eqref{regularity estimate equation} and \eqref{regularity estimate equation extension}, and $\partial_\rho \xi_\e(x)=-\alpha \xi_\e(x)$ for $x\in\partial B$, we know
	$$
	\begin{aligned}
		|R_{2,1}(\e)|&=\left|\int_{\partial B}({\widetilde{u}_{\varepsilon}}(x)^2-{\widetilde{u}_{\varepsilon}}(\phi_\e(x))^2)\,d\Hn\right|\\
		&\leq 2\e\left|\int_{\partial B}{\widetilde{u}_{\varepsilon}}(x)\partial_\rho{\widetilde{u}_{\varepsilon}}(x)\psi(x)\,d\Hn\right|+O(\e^2)\\
		&=2\e\left|\int_{\partial B}v_\e(x)\partial_\rho v_\e(x)\psi(x)\,d\Hn\right|+O(\e^2)\\
		&\leq 2\e\left|\int_{\partial B}\xi_\e(x)\partial_\rho \xi_\e(x)\psi(x)\,d\Hn\right|+C\e\omega(\e)+O(\e^2)\\
		&=2\e\left|-\alpha\int_{\partial B}\xi_\e(x)^2\psi(x)\,d\Hn\right|+C\e\omega(\e)+O(\e^2)\\
		&\leq C\e\omega(\e),
	\end{aligned}
	$$
	thanks to \eqref{assumption orthogonal equation} again. Recalling the definition \eqref{jacobian} of the Jacobian $J_\e$, one has
	$$
	J_\e(x)=1+(n-1)\e\psi(x)+O(\e^2),\quad\text{for }x\in\partial B.
	$$
	Hence by using \eqref{assumption orthogonal equation}, one obtains
	$$
	\begin{aligned}
		|R_{2,2}(\e)|&=\left|\int_{\partial B}{\widetilde{u}_{\varepsilon}}(\phi_\e(x))^2(1-J_\e(x))\,d\Hn\right|\\
		&\leq (n-1)\e\left|\int_{\partial B}{\widetilde{u}_{\varepsilon}}(\phi_\e(x))^2\psi(x)\,d\Hn\right|+O(\e^2)\\
		&= (n-1)\e\left|\int_{\partial B}v_\e(x)^2\psi(x)\,d\Hn\right|+O(\e^2)\\
		&\leq (n-1)\e\left|\int_{\partial B}\xi_\e(x)^2\psi(x)\,d\Hn\right|+C\e\omega(\e)+O(\e^2)\\
		&\leq C\e\omega(\e).
	\end{aligned}
	$$
	Finally, we deal with the last term. For $x\in\Oe\backslash B$, by the uniform estimates \eqref{regularity estimate equation}, we know
	$$
	\begin{aligned}
		\ue(x)^2=\ue(\phi_\e(x/|x|))^2+O(\e)
		=\ue(x/|x|)^2+O(\e).
	\end{aligned}
	$$
	Since $|\Omega_\e\backslash B|=O(\e)$, we have
	$$
	\begin{aligned}
		\int_{\Oe\backslash B}\ue(x)^2\,dx&=\int_{\Oe\backslash B}\ue(x/|x|)^2\,dx+O(\e^2)\\
		&=\varepsilon\int_{\partial B\cap\{\psi>0\}}\ue(x)^2\psi(x)\,d\Hn+O(\e^2)\\
		&=\varepsilon\int_{\partial B\cap\{\psi>0\}}v_\e(x)^2\psi(x)\,d\Hn+O(\e^2).
	\end{aligned}
	$$
	Similarly, one can deduce that
	$$
	\begin{aligned}
		\int_{B\backslash\Oe}{\widetilde{u}_{\varepsilon}}(x)^2\,dx&=-\varepsilon\int_{\partial B\cap\{\psi<0\}}v_\e(x)^2\psi(x)\,d\Hn+O(\e^2).
	\end{aligned}
	$$
	By the above two estimates and \eqref{assumption orthogonal equation}, we know
	$$
	\begin{aligned}
		|R_{3}(\e)|&=\left|\int_{B\backslash\Oe}{\widetilde{u}_{\varepsilon}}(x)^2\,dx- \int_{\Oe\backslash B}\ue(x)^2\,dx\right|\\
		&=\left|\varepsilon\int_{\partial B}v_\e(x)^2\psi(x)\,d\Hn+O(\e^2)\right|\\
		&\leq \varepsilon\left|\int_{\partial B}\xi_\e(x)^2\psi(x)\,d\Hn\right|+C\e\omega(\e)+O(\e^2)\\
		&\leq C\e\omega(\e).
	\end{aligned}
	$$
	The proof is complete.
\end{proof}
\subsection{A fine control of the $C^1$ distance between $v_\e$ and the eigenspace corresponding to $\lambda_2(B;\alpha)$}
From Lemma \ref{decay rate}, to improve the decays of $R_i(\e)$, $i=1,2,3$, we need a fine control of the $C^1$ distance between $v_\e$ and the eigenspace corresponding to $\lambda_2(B;\alpha)$. This will be done in the current subsection.\vspace{0.5em}

Firstly, we present a $W^{1,2}$ approximation, which is easier to obtain.
\begin{lemma}\label{bijin w12}
	For every $0<\e\ll1$, there exists an eigenfunction $\xi_\e$ relative to $\lambda_2(B;\alpha)$, such that
	\begin{equation}\label{approximation W12}
		\|v_\e-\xi_\e\|_{W^{1,2}(B)}\leq C\sqrt{|R_1(\e)|+|R_2(\e)|+|R_3(\e)|}+C\e,
	\end{equation}
	for some positive constant $C>0$ independent of $\e$.
\end{lemma}
\begin{proof}
	The idea is similar as the one proposed in \cite{Brasco=GAFA}. Let $\{\xi_k\}_{k=1}^\infty$ be the normalized Robin eigenfunctions with identity $L^2$ norm, it is well known that $\{\xi_k\}_{k=1}^\infty$ is an orthonormal basis of $L^2(B)$. Define the Robin bilinear form $\mathcal{R}:W^{1,2}(B)\times W^{1,2}(B)\to\R$ by
	$$
	\mathcal{R}(u,v)\coloneqq\int_{B}\langle\nabla u(x),\nabla v(x)\rangle\,dx+\alpha\int_{\partial B}u(x)v(x)\,d\Hn,
	$$
	then we have 
	$$
	\mathcal{R}(\xi_i,\xi_j)=\sqrt{\lambda_i(B;\alpha)\lambda_j(B;\alpha)}\delta_{ij},\quad\text{for }i,j\in\mathbb{N}_+,
	$$
	where $\delta_{ij}$ is the Kronecker Delta. Since \eqref{veu1} says that $\int_{B}v_\e \xi_1dx=0$, we have the following spectral decomposition
	$$
	v_\e(x)=\sum_{k=2}^\infty a_k(\e)\xi_k(x),\quad a_k(\e)=\int_{B}v_\e(x)\xi_k(x)\,dx.
	$$
	Also, we know
	$$
	\|v_\e\|_{L^2(B)}^2=\sum_{k=2}^\infty a_k(\e)^2.
	$$
Consequently, by \eqref{normalized} and \eqref{youyong},
$$
\begin{aligned}
	\left|\sum_{k=2}^\infty a_k(\e)^2-1\right|
	&=\left|\int_{B}v_\e(x)^2\,dx-\int_{\Oe}\ue(x)^2\,dx\right|\\
	&=\left|\int_{B}(v_{\e}(x)^2-{\widetilde{u}_{\varepsilon}}(x)^2)\,dx+R_3(\e)\right|\\
	&\leq C\e^2+|R_3(\e)|.
\end{aligned}
$$
We rewrite the above inequality as follows
\begin{equation}\label{budengshi yi}
	\left|\sum_{k=2}^{n+1}a_k(\e)^2-1\right|\leq\sum_{k=n+2}^\infty a_k(\e)^2+C\e^2+|R_3(\e)|.
\end{equation}
On the other hand, we know by \eqref{youyong2} and Lemma \ref{upper bound} that
$$
\begin{aligned}
	&\quad\left|\int_{B}|\nabla v_\e(x)|^2\,dx+\alpha\int_{\partial B}v_\e(x)^2\,d\Hn-\lambda_2(B;\alpha)\right|\\
	&\leq \left|\int_{\Oe}|\nabla \ue(x)|^2\,dx+\alpha\int_{\partial \Oe}\ue(x)^2\,d\Hn-\lambda_2(B;\alpha)\right|+C\e^2+|R_1(\e)|+|R_2(\e)|\\
	&=|\lambda_2(\Oe;\alpha)-\lambda_2(B;\alpha)|+C\e^2+|R_1(\e)|+|R_2(\e)|\\
	&\leq C\e^2+C(|R_1(\e)|+|R_2(\e)|+|R_3(\e)|),
\end{aligned}
$$
which, together with $$\int_{B}|\nabla v_\e(x)|^2\,dx+\alpha\int_{\partial B}v_\e(x)^2\,d\Hn=\sum_{k=2}^\infty \lambda_k(B;\alpha)a_k(\e)^2,$$ can be rewritten as
\begin{equation}\label{budengshi er}
	\sum_{k=n+2}^\infty \lambda_k(B;\alpha)a_k(\e)^2\leq \lambda_2(B;\alpha)\left|\sum_{k=2}^{n+1} a_k(\e)^2-1\right|+C\e^2+C(|R_1(\e)|+|R_2(\e)|+|R_3(\e)|).
\end{equation}
We can now combine \eqref{budengshi yi} and \eqref{budengshi er} to obtain that
$$
\sum_{k=n+2}^\infty (\lambda_k(B;\alpha)-\lambda_2(B;\alpha))a_k(\e)^2\leq C\e^2+C(|R_1(\e)|+|R_2(\e)|+|R_3(\e)|).
$$
By Proposition \ref{second}, we know
$$
1-\frac{\lambda_2(B;\alpha)}{\lambda_k(B;\alpha)}\geq 1-\frac{\lambda_2(B;\alpha)}{\lambda_{n+2}(B;\alpha)}>0,\quad \forall k\geq n+2,
$$
is positive and non-decreasing, from which we get
$$
\sum_{k=n+2}^\infty \lambda_k(B;\alpha)a_k(\e)^2\leq C\e^2+C(|R_1(\e)|+|R_2(\e)|+|R_3(\e)|)
$$
with a different constant $C$ depending on the spectral gap $\lambda_{n+2}(B;\alpha)-\lambda_2(B;\alpha)$, but still independent of $\e$. Now, take
$$
\xi_\e(x)\coloneqq\sum_{k=2}^{n+1}a_k(\e)\xi_k(x),
$$
which is an eigenfunction relative to $\lambda_2(B;\alpha)$. We claim that $\xi_\e$ is the desired eigenfunction. Notice that
$$
v_\e-\xi_\e=\sum_{k=n+2}^{\infty}a_k(\e)\xi_k.
$$
Due to the classic trace inequality \cite[page 81]{Henrot=2017Book}, we know
$$
\begin{aligned}
	\|v_\e-\xi_\e\|_{W^{1,2}(B)}^2&\leq C\mathcal{R}(v_\e-\xi_\e,v_\e-\xi_\e)+C\|v_\e-\xi_\e\|_{L^{2}(B)}^2\\
	&=C\sum_{k=n+2}^\infty \lambda_k(B;\alpha)a_k(\e)^2+C\sum_{k=n+2}^\infty a_k(\e)^2\\
	&\leq C\e^2+C(|R_1(\e)|+|R_2(\e)|+|R_3(\e)|).
\end{aligned}
$$
The proof is complete.
\end{proof}
Now, we are aiming to establish the following $C^1$ approximation result by Lemma \ref{bijin w12} and elliptic estimates.
\begin{lemma}\label{bijin c1}
	For every $0<\e\ll1$, there exists an eigenfunction $\xi_\e$ relative to $\lambda_2(B;\alpha)$, such that
	\begin{equation}\label{approximation c1}
		\|v_\e-\xi_\e\|_{C^1(\overline{B})}\leq C\sqrt{|R_1(\e)|+|R_2(\e)|+|R_3(\e)|}+C\e,
	\end{equation}
	for some positive constant $C>0$ independent of $\e$.
\end{lemma}
\begin{proof}
	Define
	$$
	f_\e(x)\coloneqq -\Delta{\widetilde{u}_{\varepsilon}}(x),\quad x\in B.
	$$
	Notice that $v_\e$ and $\xi_\e$ are solutions for the following Neumann boundary problems
	$$
	\begin{cases}
		-\Delta v_\e=f_\e \quad &\hbox{in}\ B, \\
		\frac{\partial v_\e}{\partial \nu}=\alpha g_\e \quad &\hbox{on}\ \partial B, \\
	\end{cases}
	\quad\quad\quad\text{and}\quad\quad\quad
\begin{cases}
	-\Delta \xi_\e=\lambda_2(B;\alpha)\xi_\e \quad &\hbox{in}\ B, \\
	\frac{\partial \xi_\e}{\partial \nu}=-\alpha \xi_\e \quad &\hbox{on}\ \partial B, \\
\end{cases}
$$
where the boundary value $g_\e$ of the normal derivative of $v_\e$ is given by
$$
\begin{aligned}
	g_\e(x)&=-v_\e(x)+\left(v_\e(x)-\ue(\phi_\e(x))\right)+\frac{1}{\alpha}\langle\nabla{\widetilde{u}_{\varepsilon}}(x)-\nabla\ue(\phi_\e(x)),\nu_{\Omega_{\varepsilon}}(\phi_\e(x))\rangle\\
	&\quad+\frac{1}{\alpha}\langle\nabla{\widetilde{u}_{\varepsilon}}(x),\nu_B(x)-\nu_{\Omega_{\varepsilon}}(\phi_\e(x))\rangle\\
	&\eqqcolon-v_\e(x)+\sum_{i=1}^3g_{\e,i}(x),\quad x\in\partial B.
\end{aligned}
$$
Here, we have used the fact that $\nabla{\widetilde{u}_{\varepsilon}}=\nabla v_\e$ on $\partial B$ and $\langle\nabla\ue,\nu_{\Omega_{\varepsilon}}\rangle=-\alpha \ue$ on $\partial\Oe$. 

By the standard elliptic estimate \cite[page 438]{MR4703940} and the triangle inequality, we know that for every $k\geq2$,
$$
\begin{aligned}
	\|v_{\varepsilon}-\xi_{\varepsilon}\|_{W^{k, 2}(B)} &\leq C\left(\|f_{\varepsilon}-\lambda_2(B;\alpha)\xi_\e\|_{W^{k-2,2}(B)}+\|v_{\varepsilon}-\xi_{\varepsilon}\|_{W^{k-1,2}(B)}+\|g_{\varepsilon}+\xi_{\varepsilon}\|_{W^{k-3 / 2,2}(\partial B)}\right) \\
	&\leq C\left(\|f_{\varepsilon}-\lambda_2(B;\alpha)\xi_\e\|_{W^{k-2,2}(B)}+\|v_{\varepsilon}-\xi_{\varepsilon}\|_{W^{k-1,2}(B)}\right. \\
	&\quad+\|v_{\varepsilon}-\xi_{\varepsilon}\|_{W^{k-3 / 2,2}(\partial B)}+\sum_{i=1}^3\|g_{\e,i}\|_{W^{k-3 / 2,2}(\partial B)}\Big).
\end{aligned}
$$
Notice that for $x\in\Omega_\e\cap B$, 
$$
f_\e(x)=-\Delta \ue(x)=\lambda_2(\Oe;\alpha)\ue(x)=\lambda_2(\Oe;\alpha)(v_\e(x)+\delta).
$$
For any $x\in B$, it follows from \eqref{omega epsilon} that there exists $x^\prime\in \Oe\cap B$ such that $|x-x^\prime|\leq C\e$, thus by the uniform estimates \eqref{regularity estimate equation} and \eqref{regularity estimate equation extension}, we obtain
$$
\begin{aligned}
	f_\e(x)&=-\Delta v_\e(x)=-\Delta v_\e(x^\prime)+O(\e)\\
	&=-\Delta u_\e(x^\prime)+O(\e)\\
	&=\lambda_2(\Oe;\alpha)(v_\e(x^\prime)+\delta)+O(\e)\\
	&=\lambda_2(\Oe;\alpha)v_\e(x)+O(\e).
\end{aligned}
$$
Therefore, by Lemma \ref{upper bound} and \eqref{regularity estimate v}, one has
$$
\begin{aligned}
	\|f_{\varepsilon}-\lambda_2(B;\alpha)\xi_\e\|_{W^{k-2,2}(B)}&=\|\lambda_2(\Oe;\alpha)v_\e-\lambda_2(B;\alpha)\xi_\e+O(\e)\|_{W^{k-2,2}(B)}\\
	&\leq C|\lambda_2(\Oe;\alpha)-\lambda_2(B;\alpha)|+C\|v_\e-\xi_\e\|_{W^{k-2,2}(B)}+C\e\\
	&\leq C\sqrt{|R_1(\e)|+|R_2(\e)|+|R_3(\e)|}+C\e+C\|v_\e-\xi_\e\|_{W^{k-2,2}(B)}\\
	&\leq C\sqrt{|R_1(\e)|+|R_2(\e)|+|R_3(\e)|}+C\e+C\|v_\e-\xi_\e\|_{W^{k-1,2}(B)}.
\end{aligned}
$$
Now we handle the boundary term. Since ${\widetilde{u}_{\varepsilon}}(\phi_\e(x))={\ue}(\phi_\e(x))$ for $x\in\partial B$, we know by the uniform estimates \eqref{regularity estimate equation} and \eqref{regularity estimate equation extension} that
$$
\begin{aligned}
	\|g_{\e,1}\|_{W^{k-3 / 2,2}(\partial B)}&=\|{\widetilde{u}_{\varepsilon}}-{\widetilde{u}_{\varepsilon}}\circ\phi_\e-\delta\|_{W^{k-3 / 2,2}(\partial B)}=O(\e).
\end{aligned}
$$
Besides, it follows from the triangle inequality and uniform estimates \eqref{regularity estimate equation}, \eqref{regularity estimate equation extension} that
$$
\begin{aligned}
	\|g_{\e,2}\|_{W^{k-3 / 2,2}(\partial B)}&\leq C\|\nabla{\widetilde{u}_{\varepsilon}}-\nabla{\ue}\circ\phi_\e\|_{W^{k-3 / 2,2}(\partial B)}\\
	&\leq C\|\nabla{\widetilde{u}_{\varepsilon}}-\nabla({\ue}\circ\phi_\e)\|_{W^{k-3 / 2,2}(\partial B)}\\
	&\quad+C\|\nabla({\ue}\circ\phi_\e)-\nabla{\ue}\circ\phi_\e\|_{W^{k-3 / 2,2}(\partial B)}\\
	&\leq C\e,
\end{aligned}
$$
thanks to the Leibniz's rule and the fact that ${\widetilde{u}_{\varepsilon}}(\phi_\e(x))={\ue}(\phi_\e(x))$ for $x\in\partial B$. Finally, it was proved by \cite[page 4707]{Brasco=Jfa} that
$$
\|g_{\e,3}\|_{W^{k-3 / 2,2}(\partial B)}\leq \|\nu_B-\nu_{\Omega_{\varepsilon}}\circ\phi_\e\|_{W^{k-3 / 2,2}(\partial B)}\leq C\e.
$$
Collecting all the estimates, we are led to
\begin{equation}\label{zuizhong}
	\begin{aligned}
		\|v_{\varepsilon}-\xi_{\varepsilon}\|_{W^{k, 2}(B)}&\leq  C\sqrt{|R_1(\e)|+|R_2(\e)|+|R_3(\e)|}+C\e\\
		&\quad+C\|v_\e-\xi_\e\|_{W^{k-1,2}(B)}+C\|v_{\varepsilon}-\xi_{\varepsilon}\|_{W^{k-3 / 2,2}(\partial B)}.
	\end{aligned}
\end{equation}
Now we use a standard bootstrap procedure to obtain the $C^1$ approximation and the reader can see the choice of $k_0$ clearly here. Firstly, for $k=2$, by the trace inequality, 
$$
\|v_{\varepsilon}-\xi_{\varepsilon}\|_{W^{1 / 2,2}(\partial B)} \leq C\|v_{\varepsilon}-\xi_{\varepsilon}\|_{W^{1,2}(B)}.
$$
By \eqref{approximation W12} and \eqref{zuizhong} with $k=2$, we have
$$
\|v_{\varepsilon}-\xi_{\varepsilon}\|_{W^{2, 2}(B)}\leq C\sqrt{|R_1(\e)|+|R_2(\e)|+|R_3(\e)|}+C\e.
$$
Applying the trace inequality again, we obtain
$$
\|v_{\varepsilon}-\xi_{\varepsilon}\|_{W^{3 / 2,2}(\partial B)} \leq C\|v_{\varepsilon}-\xi_{\varepsilon}\|_{W^{2,2}(B)} \leq C\sqrt{|R_1(\e)|+|R_2(\e)|+|R_3(\e)|}+C\e.
$$
Therefore, using \eqref{zuizhong} with $k=3$, we arrive the higher order estimate
$$
	\|v_{\varepsilon}-\xi_{\varepsilon}\|_{W^{3, 2}(B)}\leq  C\sqrt{|R_1(\e)|+|R_2(\e)|+|R_3(\e)|}+C\e.
$$
Take $k_0=[n / 2]+2$. Then finitely many repetitions of the previous argument give
$$
\|v_{\varepsilon}-\xi_{\varepsilon}\|_{W^{k_0, 2}(B)}\leq  C\sqrt{|R_1(\e)|+|R_2(\e)|+|R_3(\e)|}+C\e,
$$
and then \eqref{approximation c1} follows by the classic Sobolev embedding theorem.
\end{proof}
\subsection{Conclusion}
We have finished all the technical mission and now it is time to yield the conclusion. By Lemma \ref{upper bound}, we know
\begin{equation}\label{4=71}
|\lambda_2(B;\alpha)-\lambda_2(\Omega_\e;\alpha)|\leq C(|R_1(\e)|+|R_2(\e)|+|R_3(\e)|)+C\e^2. 
\end{equation}
Take $\omega(\e)=\sqrt{|R_1(\e)|+|R_2(\e)|+|R_3(\e)|}+\e$. Then \eqref{estimate for R123} implies
$\e\leq \omega(\e)\leq C{\e}^{1/2}$.
Consequently, by applying Lemma \ref{decay rate} and Lemma \ref{bijin c1} with this choice of $\omega$, we obtain
$$
|R_1(\e)|+|R_2(\e)|+|R_3(\e)|\leq C\e\omega(\e)=C\e \sqrt{|R_1(\e)|+|R_2(\e)|+|R_3(\e)|}+C\e^2.
$$
It follows from Young's inequality that
$$
|R_1(\e)|+|R_2(\e)|+|R_3(\e)|\leq \widetilde{C}\e^2,
$$
for some $\widetilde{C}$ independent of $\e$.
Inserting this into \eqref{4=71} leads to
\begin{equation}\label{ffff}
|\lambda_2(B;\alpha)-\lambda_2(\Omega_\e;\alpha)|\leq C\e^2. 
\end{equation}
There is still a little work to do before proving the sharpness of the exponent 2. 
\begin{proof}[Proof of Theorem \ref{main result three}]
	Since $\mathcal{A}(\Oe)\simeq\e$, what we actually need is to prove that
	$$
	C_1\e^2\leq |\lambda_2(B_\e;\alpha)-\lambda_2(\Omega_\e;\alpha)|\leq C_2\e^2
	$$
	for some constant $C_1,C_2>0$ independent of $\e$, where $B_\e$ is a ball center at the origin such that $|B_\e|=|\Omega_\e|$. From the expression of the constant $\gamma(n,\alpha,R)$ in Theorem \ref{main result one}, it is easy to get the existence of the constant $C_1$. For the existence of $C_2$, recall that in the proof of Lemma \ref{upper bound}, we have used that
	$$
	\lambda_2(B_\e;\alpha)=t_\e^{-2}\lambda_2(B;t_\e \alpha),
	$$
	with
	$$
	t_\e=\left(\frac{|\Omega_\e|}{|B|}\right)^{1/n}=\left(1+O(\e^2)\right)^{1/n}=1+O(\e^2).
	$$
	Again, let $u_2$ be a normalized eigenfunction for $\lambda_2(B;\alpha)$ and $u_{2,\e}$ be a normalized eigenfunction for $\lambda_2(B;t_\e \alpha)$, we know 
	$$
	\lambda_2(B;t_\e \alpha)\leq Q[u_2; B, t_\e\alpha]= \lambda_2(B;\alpha)+(t_\e-1)\alpha\int_{\partial B}u_2(x)^2\,d\Hn
	$$
	and
	$$
	\lambda_2(B; \alpha)\leq Q[u_{2,\e}; B,\alpha]=\lambda_2(B;t_\e\alpha)+(1-t_\e)\alpha\int_{\partial B}u_{2,\e}(x)^2\,d\Hn.
	$$
	Therefore,
	$$
	\begin{aligned}
		|\lambda_2(B_\e;\alpha)-\lambda_2(B;\alpha)|&\leq t_\e^{-2}|\lambda_2(B;t_\e \alpha)-\lambda_2(B;\alpha)|+|1-t_\e^{-2}|\lambda_2(B;\alpha)\\
		&\leq C t_\e^{-2}|1-t_\e|\left|\int_{\partial B}(u_2(x)^2+u_{2,\e}(x)^2)\,d\Hn\right|+C\e^2\\
		&\leq C_3\e^2.
	\end{aligned}
	$$
This, together with \eqref{ffff}, implies the existence of $C_2>0$ such that $|\lambda_2(B_\e;\alpha)-\lambda_2(\Omega_\e;\alpha)|\leq C_2\e^2$. This completes the proof of Theorem \ref{main result three}.
\end{proof}

\section*{Acknowledgement}
Zhijie Chen is supported by National Key R\&D Program of China (Grant 2023YFA1010002) and National Natural Science Foundation of China (No. 12222109). Wenming Zou is supported by National Key R\&D Program of China (Grant 2023YFA1010001) and National Natural Science Foundation of China (No. 12171265 and 12271184).

\section*{Data availibility statement}
No data was used for the research described in the article.

\section*{Declarations}
\subsection*{Conflict of interest}
The authors declare that they have no conflict of interests, they also confirm that the manuscript complies to the Ethical Rules applicable for this journal.

	\bibliographystyle{amsplain}

\end{document}